\documentclass[12pt,a4paper, reqno]{amsart}
  \usepackage{amssymb,amsmath,amsopn,textcomp, thmtools}

  \usepackage[backgroundcolor=white, bordercolor=blue, 
linecolor=blue]{todonotes}
  \usepackage[normalem]{ulem}

  \usepackage[a4paper]{geometry}
  \usepackage{lscape}
  \usepackage[ansinew]{inputenc}
  \usepackage[english]{babel}
  \usepackage{mathrsfs}
  \usepackage{setspace}
  \usepackage{enumitem}
  \usepackage{stmaryrd}
  \usepackage[titletoc]{appendix}
  \usepackage[colorlinks=true, linkcolor=blue, citecolor=blue]{hyperref}

  \usepackage{dsfont}
  \usepackage{bbm}

  \usepackage{graphicx}
  \usepackage{color}
  \usepackage{pgf,tikz}
  \usetikzlibrary{arrows,angles,quotes}  
  \usepackage{caption,rotating}

  \newtheorem{theorem}{Theorem}[section]
  \newtheorem{corollary}[theorem]{Corollary}
  \newtheorem{proposition}[theorem]{Proposition}
  \newtheorem{lemma}[theorem]{Lemma}
  \newtheorem*{theorem*}{Theorem}

  \theoremstyle{definition}
  \newtheorem{definition}[theorem]{Definition}
  \theoremstyle{remark}
  \newtheorem{remark}[theorem]{Remark}

  \renewenvironment{proof}[1][Proof]{\noindent\textit{#1.} }{\hfill 
\rule{0.5em}{0.5em}}

  \newcommand {\N} {\mathbbm{N}}
  \newcommand {\Z} {\mathbbm{Z}}
  \newcommand {\R} {\mathbbm{R}}

  \newcommand {\eps} {\varepsilon}

  \renewcommand{\S}{\mathcal{S}}
  \renewcommand{\d}{\mathrm{d}}
  \newcommand{\dxy} {\, \d (x,y)}
  \newcommand{\dst} {\, \d (s,t)}

  \newcommand{\notion}[1]{\emph{#1}}
  \newcommand{\BBB}{B}
  \newcommand{\RRR}{\mathbbm{R}}
  \newcommand{\ZZZ}{\mathbbm{Z}}
  \newcommand{\newvariable}[2]{%
    \newcommand{#1}[1][]{#2_{##1}}%
  }
  \newvariable{\TheConf}{\Gamma}
  \newcommand{\TheConfOf}[1]{\TheConf(#1)}
  \newcommand{\mapcolon}{\colon}
  \newvariable{\DirGraph}{G}
  \newcommand{\DirGraphOf}[1]{\DirGraph(#1)}
  \newvariable{\ThePoint}{x}
  \newvariable{\WellConnectedPt}{\hat{\ThePoint}}
  \newvariable{\AltPoint}{y}
  \newvariable{\ThrPoint}{z}
  \newvariable{\TheOpenSet}{U}
  \newvariable{\TheOpenNbhd}{W}
  \newvariable{\TheSmallOpenSet}{U'}
  \newcommand{\TheDoubleCone}{V}
  \newvariable{\AltDoubleCone}{V'}
  \newcommand{\ShrinkDoubleCone}[1]{\TheDoubleCone_{#1}}
  \newcommand{\TheCone}{\widetilde{\TheDoubleCone}}
  \newcommand{\LargeCone}{\overline{\TheDoubleCone}}
  \newvariable{\SetDoubleCones}{\mathcal{V}}
  \newvariable{\TheDim}{d}
  \newvariable{\TheApex}{\vartheta}
  \newvariable{\TheSpace}{\RRR^{\TheDim}}
  \newcommand{\RestrGraph}{\DirGraph[\TheOpenSet]}
  \newcommand{\SmallRestrGraph}{\DirGraph[\TheSmallOpenSet]}
  \newcommand{\CardOf}[1]{\##1}
  \newvariable{\TheFactor}{\lambda}
  \newcommand{\NormOf}[1]{\|#1\|}
  \newcommand{\BallOf}[2][]{\BBB_{#1}(#2)}
  \newvariable{\TheRadius}{r}
  \newvariable{\TheLattice}{\ZZZ^{\TheDim}}
  \newvariable{\SmallRadius}{r}
  \newvariable{\SmallerRadius}{r'}
  \newvariable{\LargeRadius}{R}
  \newvariable{\ChromaticBound}{\rho}
  \newcommand{\SqrtOf}[1]{\sqrt{#1}}
  \newcommand{\SinOf}[1]{\operatorname{sin}(#1)}
  \newcommand{\TranslateCone}[2]{\TheDoubleCone^{#2}[#1]}
  \newcommand{\BasedCone}[1]{\TranslateCone{#1}{\TheConf}}
  \newcommand{\intersect}{\cap}
  \newvariable{\NumTypes}{k}
  \newvariable{\AltSmall}{s}
  \newvariable{\AltLarge}{S}
  \newvariable{\WellSmall}{\hat{\AltSmall}}
  \newvariable{\WellLarge}{\hat{\AltLarge}}
  \newvariable{\WellRadius}{\hat{\LargeRadius}}
  \newvariable{\JumpMax}{\delta}
  \newcommand{\Enlarge}[1]{#1+\frac{2#1+\SqrtOf{\TheDim}}{\SinOf{\TheApex}}}
  \newvariable{\Cube}{A}
  \newcommand{\CubeOf}[2][]{\Cube[#1](#2)}
  \newvariable{\TheBlock}{Q}
  \newvariable{\AltBlock}{P}
  \newvariable{\Block}{{\TheBlock}}
  \newcommand{\BlockOf}[2][]{\Block[#1](#2)}
  \newvariable{\TheLength}{\ell}
  \newvariable{\TheBall}{B}
  \newcommand{\InfNormOf}[1]{\|#1\|_{\infty}}
  \newcommand{\SetOf}[2][]{%
    \left\{
      #1 \,|\, #2
    \right\}%
  }
  \newcommand{\FamOf}[2][]{\left(#1\right)_{#2}}
  \newvariable{\TheScale}{h}
  \newvariable{\TheTown}{T}
  \newvariable{\Town}{{\TheTown}}
  \newvariable{\TownConf}{W}
  \newcommand{\TownConfOf}[2][]{\TownConf[#1](#2)}
  \newcommand{\TownOf}[2][]{\Town[#1](#2)}
  \newvariable{\TheCubeDistance}{\delta}
  \newcommand{\TheCubeDistanceOf}[2][]{\TheCubeDistance[#1](#2)}
  \newvariable{\ThePath}{p}
  \newvariable{\BdNumJumps}{N}
  \newvariable{\BdEdgeUsage}{M}
  \newcommand{\DistOf}[2]{|#1-#2|}
  \newvariable{\TheScaleStep}{\Delta}
  \newvariable{\LogScale}{n}
  \newcommand{\TheScaleLog}{\TheScaleStep^{\LogScale}}
  \newcommand{\TheScaleLogmin}{\TheScaleStep^{\LogScale-1}}
  \newvariable{\FirstJump}{{R_1}}
  \newvariable{\LengthChain}{l}
  \newvariable{\TheVoter}{v}
  \newvariable{\TheIndex}{i}
  \newcommand{\minDist}{R_0}
  \newvariable{\StepComparability}{\lambda}
  \newvariable{\InfApex}{\vartheta}
  \newcommand{\squish}[1]{\null\negthinspace#1\negthinspace\null}
  \newcommand{\squash}[1]{\squish{\squish{\squish{\squish{#1}}}}}
  \newcommand{\Squash}[1]{\squish{\squish{\squish{\squish{\squish{\squish{\squish{#1}}}}}}}}
  \newcommand{\qcolon}{\,:\,}
  
  \allowdisplaybreaks
  \addto\extrasenglish{}
  \addto\extrasenglish{} 
\begin{document}

\title{Quadratic forms and Sobolev spaces of fractional order}

  \author{Kai-Uwe Bux}
  \email{bux\_{}math@kubux.net}
  \address{Universit\"{a}t Bielefeld, Fakult\"{a}t f\"{u}r Mathematik, Postfach 
  100131, D-33501 Bielefeld, Germany}
  
  \author{Moritz Kassmann}
  \email{moritz.kassmann@uni-bielefeld.de}
  \address{Universit\"{a}t Bielefeld, Fakult\"{a}t f\"{u}r Mathematik, Postfach 
  100131, D-33501 Bielefeld, Germany}
  
  \author{Tim Schulze}
  \email{tschulze@math.uni-bielefeld.de}
  \address{Universit\"{a}t Bielefeld, Fakult\"{a}t f\"{u}r Mathematik, 
  Postfach 
  100131, D-33501 Bielefeld, Germany}
  


\date{\today}

\thanks{Financial support by the DFG through the CRCs 701 and 1283 is acknowledged.}

  \begin{abstract}
    We study quadratic functionals on $L^2(\R^d)$ that generate seminorms in 
the fractional Sobolev space $H^s(\R^d)$ for $0 < s < 1$. The functionals under 
consideration appear in the study of Markov jump processes and, 
independently, in recent research on the Boltzmann equation. The functional 
measures differentiability of a function $f$ in a similar way as the seminorm of 
$H^s(\R^d)$. The major difference is that differences $f(y) - f(x)$ are taken 
into account only if $y$ lies in some double cone with apex at $x$ or vice 
versa. The configuration of double   cones is allowed to be inhomogeneous 
without any assumption on the spatial regularity. We prove that the resulting 
seminorm is comparable to the standard one of $H^s(\R^d)$. The proof follows 
from a similar result on discrete quadratic forms in $\Z^d$, which is our second 
main result. We establish a general scheme for discrete 
approximations of nonlocal quadratic forms. Applications to Markov jump 
processes are discussed. 
  \end{abstract}
  
  \maketitle

\section{Introduction}\label{sec:intro}

The Sobolev-Slobodecki\u{\i} space $H^s(\R^d)$, $0 < s < 1$, can be defined as 
the set of all functions $f \in L^2(\R^d)$ such that the seminorm
\begin{align}\label{eq:seminorm_Hs}
\int\limits_{\R^d \times \R^d} \big(f(y)-f(x) \big)^2 |x-y|^{-d-2s} \dxy 
\end{align}
is finite, see the original work \cite{Slo58} or the monographs 
\cite{RuSi96, AdFo03, Agr15}. The normed space is complete and, together with 
its modifications for domains in $\R^d$, is of fundamental importance in the 
field of Partial Differential Equations. In this article, we adopt 
the common notation from Stochastic Analysis, where $H^s(\R^d)$ usually is 
denoted by $H^{\alpha/2}(\R^d)$ with $\alpha = 2s \in (0,2)$. The 
corresponding stochastic process is called $\alpha$-stable process, which 
explains the usage of $\alpha$ here. 

We study seminorms on $L^2(\R^d)$, which are very similar 
but smaller than \eqref{eq:seminorm_Hs} because we consider differences $f(y) - 
f(x)$ only if $y$ lies in some double cone with apex at $x$. Below, we explain
where and why the corresponding quadratic forms appear naturally. In order to 
formulate our main result, let us fix some notation. By $V$ 
we denote a double cone in $\R^d$ with apex at $0 \in \R^d$, symmetry axis $v 
\in \R^d$, and apex angle $\vartheta \in (0,\frac{\pi}{2}]$. Let 
$\mathcal{V}= (0,\frac{\pi}{2}] \times \mathbb{P}_\R^{d-1}$ 
denote the family of all such double cones. For $x 
\in \R^d$, we define a shifted double cone by $V[x] = V + x$. A mapping 
$\Gamma: \R^d \to \mathcal{V}$ is called a configuration. If $\Gamma$ is a 
configuration with the property that the infimum $\InfApex$ over all apex 
angles 
of cones in $\Gamma(\R^d)$ is positive, then $\Gamma$ is called 
$\InfApex$-bounded. If in addition
\begin{align}\tag{M}\label{eq:mbc}
  \{\,
    (x,y) \in \R^d \times \R^d 
    \,\,|\,\,
    y-x \in \Gamma(x)
  \,\} 
\text{ is a Borel set in }\R^d \times \R^d,
\end{align}
then $\Gamma$ is called 
$\InfApex$-admissible.
For $x \in \R^d$ and $\Gamma$ a configuration, we define
$V^\Gamma [x] = x + \Gamma(x)$.

One of our main results is the following theorem.

\begin{theorem}\label{theo:main} Let $\Gamma$ be a $\InfApex$-admissible 
configuration and $\alpha \in (0,2)$. Let $k:\R^d \times \R^d \to [0,\infty]$ 
be a measurable function satisfying $k(x,y) = k(y,x)$ and 
\begin{align}\label{assum:main}
  \Lambda^{-1}
  \big( \mathbbm{1}_{V^\Gamma[x]}(y) + \mathbbm{1}_{V^\Gamma[y]}(x)\big) |x-y|^{-d-\alpha}
  \,\,\leq\,\,
  k(x,y)
  \,\,\leq\,\,
  \Lambda |x-y|^{-d-\alpha},
\end{align}
for all $x$ and $y$, where $\Lambda \geq 1$ is some constant.
Then there is a constant $c \geq 1$ such  that for every
ball $B \subset \R^d$ and for every $f \in L^2(B)$, the inequality
\begin{align}\label{eq:main-result}
  \int\limits_{B \times B} (f(x)-f(y))^2|x-y|^{-d-\alpha} \dxy 
  \,\,\leq\,\,
  c \int\limits_{B \times B} (f(x)-f(y))^2k(x,y) \dxy
\end{align}
holds.

The constant $c$ depends on $\Lambda$, the dimension $d$ and $\InfApex$. It is 
independent of $k$ and $\Gamma$.  For $0 < 
\alpha_0 \leq \alpha <2$, the constant $c$ depends on $\alpha_0$ but not on 
$\alpha$. 
\end{theorem}

Note that the reverse inequality in \eqref{eq:main-result} trivially holds 
true. 
\begin{remark}\label{comment on mbc} One strength of 
the theorem is that there are only two essential assumptions, namely, that the 
infimum of the apex angles of double cones is required to be positive and that 
the set 
\[ 
  \{\,
    (x,y) \in \R^d \times \R^d 
    \,\,|\,\,
    y-x \in \Gamma(x)
  \,\}
\]
is a Borel set in $\R^d \times \R^d$. Other than that, the 
symmetry axis of the double cone and the apex angle might 
depend on the center in an arbitrary way. Note that the last condition 
\eqref{assum: measurability} is nothing else but 
the measurability of the function $v:\R^d \times \R^d \to \R, 
v(x,y)=\mathbbm{1}_{V^\Gamma[x]}(y)$, which is important in light of 
\eqref{assum:main}.
\end{remark}

A similar result like \autoref{theo:main} has recently been provided in 
\cite[Lemma A.6]{ImSi16}. One difference between the two results is that 
\autoref{theo:main} provides comparability on every ball. This property is 
important for applications, 
e.g., for regularity results, cf. \cite[Condition (A)]{DyKa15}, and when 
studying function spaces over bounded sets, cf. 
\autoref{cor:H-are-equal}. Another difference concerns the class of cones 
considered. In our setup, 
it is generally not true that two double cones $x + \Gamma(x)$ and $y + 
\Gamma(y)$ have a nonempty intersection. This is different in the framework of 
\cite{ImSi16}, cf. Lemma A.5 therein. On the other hand, we consider classical 
double cones and not more general union of rays. 

The proof of our main result is based on discrete 
approximations of the quadratic forms involved. We establish a general scheme 
of how to approximate a given nonlocal quadratic form on $L^2(\R^d)$ through a 
sequence of discrete quadratic forms. We provide a discrete analog of 
\autoref{theo:main} that implies \autoref{theo:main}. We hope the discrete 
result itself to be useful in different fields, e.g., when studying random 
walks in inhomogeneous or random media. Let us formulate our main result in the 
discrete setup.

\begin{theorem}\label{theo:main_discrete} Let $\Gamma$ be a $\InfApex$-bounded 
configuration and $\alpha \in (0,2)$.
Let $\omega:\Z^d \times \Z^d \to [0,\infty]$ be a function 
satisfying $\omega(x,y) = \omega(y,x)$ and 
\begin{align}\label{assum:main_discrete}
  \Lambda^{-1}
    \big( \mathbbm{1}_{V^\Gamma [x]}(y) + \mathbbm{1}_{V^\Gamma [y]}(x)\big)
    |x-y|^{-d-\alpha}
  \,\,\leq\,\,
  \omega(x,y)
  \,\,\leq\,\, 
  \Lambda |x-y|^{-d-\alpha} 
\end{align}
for $|x-y| >\minDist$, where $\minDist>0$, $\Lambda \geq 1$ are some 
constants. There exist constants $\kappa \geq 1 , c \geq 1$
such that for 
every $R>0$, $x_0\in \R^d$ and every function $f:(B_{\kappa R}(x_0) \cap 
\Z^d) \to \R$, the inequality
\begin{align*}
  \sum\limits_{\Squash{\underset{|x-y|>\minDist}{x,y \in B_R(x_0) \cap  \Z^d}}}
    (f(x)-f(y))^2 |x-y|^{-d-\alpha}
  \,\,\leq\,\,
  c \sum\limits_{\Squash{\underset{\DistOf{x}{y}>\minDist}{x, y\in B_{\kappa R}(x_0)\cap  \Z^d}}}
    (f(x)-f(y))^2 \omega(x,y)
\end{align*}
holds.

The constant $c$ depends on $\Lambda, \InfApex, R_0$ and on the dimension $d$. 
It does not depend on $\omega$ and $\Gamma$.
\end{theorem}

Let us present the motivation for \autoref{theo:main} and provide some 
applications.

One motivation for our research stems from recent 
contributions to the study of the Boltzmann equation. \cite{Sil15} 
provides an approach to the Boltzmann equation, which makes use of recent 
regularity results for integrodifferential operators. This approach works 
without imposing cut-off conditions on the collision kernel. It turns out that 
this approach leads to integrodifferential operators with kernels $k(x,y)$ 
similar to those that we study here. \cite{ImSi16} develops a regularity 
theory to the  Boltzmann equation based on the approach of \cite{Sil15}. 

The first application concerns function spaces. For a domain 
$\Omega \subset \R^d$, we consider the Hilbert space
\[
  H_k(\Omega)
  \,\,=\,\,
  \big\{\,
    f \in L^2({\Omega})
    \,\,\big|\,\,
    |f|_{H_k(\Omega)} < \infty
  \,\big \}
\]
where the  seminorm $|f|_{H_k(\Omega)}$ is given by 
\[
  |f|_{H_k(\Omega)}
  \,\,=\,\,
  \int\limits_{\squash{\Omega \times \Omega}} (f(y)-f(x))^2 k(x,y) 
  \dxy.
\]
We endow $H_k(\Omega)$ with the norm $\|f\|_{H_k(\Omega)}$,
$\|f\|_{H_k(\Omega)}^2 = 
\|f\|_{L^2(\Omega)}^2 + |f|^2_{H_k(\Omega)}$. The 
space $H^{\frac{\alpha}{2}}(\Omega)$ is defined as the Hilbert space of all $f 
\in L^2(\Omega)$ such that the seminorm 
$|f|_{H^\frac{\alpha}{2}(\Omega)}$ is finite. We denote the norm 
on $H^{\frac{\alpha}{2}}(\Omega)$ by $\|f\|_{H^{\frac{\alpha}{2}}(\Omega)}$. 
Note that   
$|\cdot|_{H^{\frac{\alpha}{2}}(\Omega)}$ dominates $|\cdot|_{H_k(\Omega)}$ 
because of
  \eqref{assum:main}. Hence
  $\|\cdot\|_{H^{\frac{\alpha}{2}}(\Omega)}$ dominates 
$\|\cdot\|_{H_k(\Omega)}$ and we can deduce the
  following inclusion:
  \begin{equation}\label{eq:one-direction}
    H^{\frac{\alpha}{2}}(\Omega) \subset H_k(\Omega) \,.
  \end{equation}

As we will show, \autoref{theo:main_discrete} implies the reverse implication 
if $\Omega$ is a bounded Lipschitz domain. 
We will prove the following result in \autoref{sec:application_discrete}. 

\begin{theorem}\label{cor:H-are-equal} 
Let $\Omega \subset \R^d$ be a 
bounded Lipschitz domain. Then $H_k(\Omega) = H^{\frac{\alpha}{2}}(\Omega)$. 
The seminorms 
$|\cdot|_{H_k(\Omega)}$ and 
$|\cdot|_{H^{\frac{\alpha}{2}}(\Omega)}$ and the corresponding norms are 
comparable 
on $H_k(\Omega)$.
Moreover, the subspace $C^\infty(\overline{\Omega})$ is dense in 
$H_k(\Omega)$.

In addition $H_k(\R^d)=H^{\frac{\alpha}{2}}(\R^d)$ and the 
seminorms 
$|\cdot|_{H_k(\R^d)}$ and 
$|\cdot|_{H^{\frac{\alpha}{2}}(\R^d)}$ and the corresponding norms are 
comparable on $H_k(\R^d)$. The 
subspace 
$C_c^\infty(\R^d)$ of smooth functions with compact support in $\R^d$ is dense 
in $H_k(\R^d)$.
\end{theorem}

As mentioned above, \autoref{theo:main} has direct significance for the 
theory of Markov jump processes. 
Let us recall that a bilinear symmetric closed form $(\mathcal{E}, 
\mathcal{D}(\mathcal{E}))$ on $L^2(\R^d)$ is called Dirichlet form if it is 
Markovian, e.g., if for every $u \in  \mathcal{D}(\mathcal{E})$ the function 
$v= (u \wedge 1) \vee 0$ belongs to $\mathcal{D}(\mathcal{E})$ and satisfies 
$\mathcal{E}(v,v) \leq \mathcal{E}(u,u)$. See \cite[Section 1.1]{FOT94} for 
this definition plus comments and examples. 
A Dirichlet form $(\mathcal{E}, \mathcal{D}(\mathcal{E}))$ on $L^2(\R^d)$ is 
called regular if $C_c(\R^d) \cap \mathcal{D}(\mathcal{E})$ is dense in 
$C_c(\R^d)$ w.r.t. the supremum norm as well as in $\mathcal{D}(\mathcal{E})$ 
w.r.t. the norm $(\mathcal{E}(u,u) + (u,u))^\frac{1}{2}$. A major result is 
that every regular Dirichlet form $(\mathcal{E}, 
\mathcal{D}(\mathcal{E}))$ on $L^2(\R^d)$ corresponds to a symmetric 
strong Markov process on $(\R^d, \mathcal{B}(\R^d))$, whose Dirichlet form is 
given by $(\mathcal{E}, \mathcal{D}(\mathcal{E}))$, cf. \cite[Theorem 
7.2.1]{FOT94}. Note that the 
rotationally symmetric $\alpha$-stable process is the strong Markov process 
that corresponds to the regular Dirichlet form $(\mathcal{E}^\alpha, 
H^{\alpha/2}(\R^d))$ on $L^2(\R^d)$, where 
\begin{align*}
  \mathcal{E}^\alpha(f,g)
  \,\,=\,\,
  \int\limits_{\Squash{\R^d \times \R^d}}
  \left(f(y)-f(x) \right) \left(g(y)-g(x) \right) \; |x-y|^{-d-\alpha} \dxy 
  \,.
\end{align*} 
\autoref{theo:main} immediately implies the following result.

\begin{corollary}
The Dirichlet form $(\mathcal{E}, \mathcal{F})$ on 
$L^2(\R^d)$ with $\mathcal{F} = H^{\alpha/2}(\R^d)$ and 
\begin{align*}
  \mathcal{E} (f,g)
  \,\,=\,\,
  \int\limits_{\Squash{\R^d \times \R^d}} \left(f(y)-f(x) \right) \left(g(y)-g(x) \right)
    k(x,y) \dxy \,,
\end{align*}
is a regular Dirichlet form on $L^2(\R^d)$. There exists a corresponding 
strong Markov process.
\end{corollary}

The corresponding stochastic process is an interesting subject for further 
research. Presumably, it shares several properties with the related 
rotationally symmetric $\alpha$-stable process. Establishing sharp pointwise 
heat kernel estimates and, if applicable, the Feller property 
constitute interesting but challenging tasks.

Another application concerns regularity of solutions to integrodifferential 
equations. We can apply recent results of \cite{DyKa15} and establish a weak 
Harnack inequality and H\"{o}lder a priori estimates to corresponding weak 
solutions. 

\begin{corollary}\label{cor:regularity}
Assume $\alpha \in (0,2)$, $k(x,y)$ is as in \autoref{theo:main},  
$\Omega \subset \R^d$ is open and $f \in L^{q/\alpha}(\Omega)$ for $q > d$. 
Then every weak solution $u : \R^d \to \R$ to 
\[
  \lim\limits_{\eps \to 0+} \int\limits_{\squash{\R^d \setminus B_\eps(x)}}
  \big( u(y) - u(x) \big) k(x,y) \d y
  \,\,=\,\,
  f \qquad (x \in \Omega) \,,
\] 
satisfies a weak Harnack inequality and is H\"older regular in the 
interior of $\Omega$. 
\end{corollary}
 
The proof uses the regularity result of \cite{DyKa15}. It relies on 
\autoref{theo:main}, which ensures that condition (A) of \cite{DyKa15} is 
satisfied. Condition (B) is easily verified for the classical choice of 
Lipschitz continuous cutoff-functions.  

This article is organized as follows. In \autoref{sec:setup} we provide the 
technical definitions and explain the set-up in detail. 
In \autoref{sec:application_discrete} we explain how \autoref{theo:main} is 
derived from \autoref{theo:main_discrete}. To this end, we formulate a rescaled 
version of \autoref{theo:main_discrete} on $h \Z^d$ for $h > 0$, 
\autoref{cor:main_discrete}, and consider the limit procedure $h \searrow 0$. 
We also provide the proof of \autoref{cor:H-are-equal}. In 
\autoref{sec:proof_discrete} we provide the main tool for the proof of 
\autoref{theo:main_discrete}, which is \autoref{theo:path-props}. Since the 
main ideas can be better communicated when working in the Euclidean 
space rather than the integral lattice, we present this case separately
in \autoref{sec:prelude}. 
\autoref{sec:proof_main-result} finally contains the proof of 
\autoref{theo:main_discrete}.

\section{Set-up, definitions and preliminaries}\label{sec:setup}

The aim of this section to provide the framework of \autoref{theo:main} and 
auxiliary results needed to deduce \autoref{theo:main} from 
\autoref{theo:main_discrete}.

\begin{definition}\label{def:start}
Given $v \in \S^{d-1}$ and $\vartheta \in (0,\frac{\pi}{2}]$ we define a 
\emph{cone} by 
\begin{align*}
  \widetilde{V}
  \,\,=\,\,
  \widetilde{V}(v, \vartheta)
  \,\,=\,\,
  \Big\{\,
    h \in \R^d
    \ \Big|\
    h \ne 0,\, \frac{\langle v,h \rangle}{|h|} > \cos(\vartheta)
  \,\Big\}\,. 
\end{align*}
Let $\widetilde{\mathcal{V}}$ denote the family of all cones.
We denote the corresponding \emph{double cone} by $V$, i.e.
\begin{align*}
  V
  \,\,=\,\,
  V(v, \vartheta)
  \,\,=\,\,
  \widetilde{V} \cup (- \widetilde{V} ) \,.
\end{align*}
The set $\mathcal{V}$ of all double cones is simply the manifold $ 
(0,\frac{\pi}{2}] \times 
\mathbb{P}_\R^{d-1}$, where $\mathbb{P}_\R^{d-1}$ is the real projective space 
of dimension $d-1$. For $x 
\in \R^d$ we define a \emph{shifted cone} by $\widetilde{V}[x] = 
\widetilde{V} + 
x$ and a \emph{shifted double cone} by $V[x] = V + x$. \newline
For $r>0$ we 
call the set $V_r = V_r(v,\vartheta) = \{ y \in V \, | \,  \overline{B_r(y)} 
\subset V\}$ 
a \emph{double half-cone}. A mapping $\Gamma: \R^d \to \mathcal{V}$ is 
called \emph{configuration}. If $\Gamma$ is a configuration with the property 
that the infimum $\vartheta$ over all apex angles of cones in $\Gamma(\R^d)$ is 
positive, then $\Gamma$ is called \emph{$\InfApex$-bounded}. If $\Gamma$ is a 
$\InfApex$-bounded configuration and 
\begin{align}\label{assum: measurability}\tag{M} 
  \big\{\,
    (x,y) \in \R^d \times \R^d 
    \,\,\big|\,\,
    y-x \in \Gamma(x)
  \,\big\} 
\text{ is a Borel set in }\R^d \times \R^d,
\end{align}
then $\Gamma$ is called 
\emph{$\vartheta$-admissible}, cf. \autoref{comment on mbc}. For $x 
\in \R^d$ and $\Gamma$ a configuration, we define $V^\Gamma [x] = 
x + \Gamma(x)$ and
analogously for $r>0$
\[
  V_r^\Gamma[x] \,\,=\,\,
  \big\{\,
    y\in V^\Gamma[x]
    \,\,\big|\,\, \overline{B_r(y)} \subset V^\Gamma[x]
  \,\big\} \,.
\]
\end{definition}

\begin{figure}
\centering
  \ProvidesFile{fig--cone-and-double-half-cone.tex}
\begingroup
  \begin{tikzpicture}[scale=1.5]
      \coordinate (O) at (0,0) {};
      \coordinate (B) at (1,2) {};
      \coordinate (C) at (-1,2) {};
      \coordinate (D) at (1,-2) {};
      \coordinate (E) at (-1,-2) {};
      \coordinate (F) at (0,-1.5);
      \coordinate [label = above:$v$](G) at (0,1.5) {};
      \draw (E) -- (B);
      \draw (C) -- (D);
      \draw[dashed] (F) -- (G);
      \pic[draw=black, angle eccentricity=0.65, angle radius = 1.5cm, pic text={$\vartheta$}]
    {angle=B--O--G}; 
  \end{tikzpicture}
  \hspace{2cm}
  \begin{tikzpicture}[scale=1.5]
      \coordinate (O) at (0,0) {};
      \coordinate (B) at (1,2) {};
      \coordinate (C) at (-1,2) {};
      \coordinate (D) at (1,-2) {};
      \coordinate (E) at (-1,-2) {};
      \coordinate (F) at (0,-1.5);
      \coordinate (G) at (0,1.5) {};
      \coordinate (H) at (0,1);
      \coordinate (I) at (0.5,2);
      \coordinate (J) at (-0.5,2);
      \coordinate (H') at (0,-1);
      \coordinate (I') at (0.5,-2);
      \coordinate (J') at (-0.5,-2);
      \coordinate (r1) at (0.5,1);
      \coordinate (r2) at (0.1,1.2);
      \draw [dashed] (E) -- (B);
      \draw [dashed] (C) -- (D); 
      \draw[blue] (I) -- (H) -- (J);
      \draw[blue] (I') -- (H') -- (J');
      \draw[thick, dashed, red] (r1) -- (r2) node [midway, above, sloped, red] (TextNode) {$r$};
  \end{tikzpicture}
\endgroup
 
\caption{Example of a cone $V(v,\vartheta)$ and a double half-cone 
$V_r(v,\vartheta)$ for $d=2$}
\end{figure}
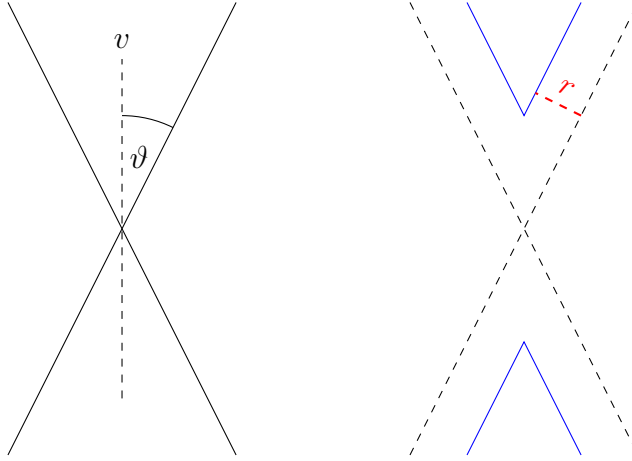

\definecolor{cqcqcq}{rgb}{0.75,0.75,0.75}
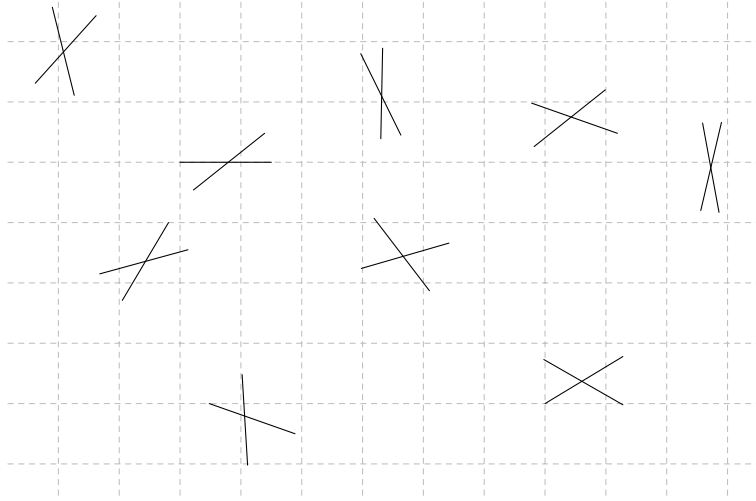
\begin{figure}
\begin{center}
  \ProvidesFile{fig--configuration.tex}
\begingroup
  \begin{tikzpicture}[scale=0.8, line cap=round,line join=round,>=triangle 45,x=1.0cm,y=1.0cm]
    \draw [color=cqcqcq,dash pattern=on 2pt off 2pt, xstep=1.0cm,ystep=1.0cm] (-4.83,-5.52) grid (7.64,2.69);
    \clip(-4.83,-5.52) rectangle (7.64,2.69);
    \draw (-2,0)-- (-0.5,0);
    \draw (-1.78,-0.46)-- (-0.61,0.48);
    \draw (-2.19,-1)-- (-2.95,-2.29);
    \draw (0.98,-1.76)-- (2.42,-1.34);
    \draw (1.19,-0.93)-- (2.1,-2.13);
    \draw (-3.32,-1.85)-- (-1.87,-1.45);
    \draw (3.78,0.98)-- (5.19,0.48);
    \draw (3.82,0.26)-- (4.99,1.2);
    \draw (1.63,0.45)-- (0.97,1.8);
    \draw (1.3,0.39)-- (1.33,1.89);
    \draw (4,-4)-- (5.28,-3.22);
    \draw (5.28,-4.02)-- (3.98,-3.27);
    \draw (6.56,-0.8)-- (6.9,0.66);
    \draw (6.86,-0.83)-- (6.59,0.65);
    \draw (-0.89,-5.02)-- (-0.98,-3.52);
    \draw (-1.52,-4)-- (-0.11,-4.5);
    \draw (-4.38,1.31)-- (-3.38,2.43);
    \draw (-3.74,1.11)-- (-4.1,2.57);
  \end{tikzpicture}
\endgroup
 
\end{center}
\captionof{figure}{A possible configuration $\Gamma$}
\label{fig:configuration}
\end{figure}

One key observation of our approach is that the large, possibly uncountable, 
family of cones generated by a $\vartheta$-admissible configuration $\Gamma$ 
can 
be reduced to a 
finite family of 
cones.

\begin{lemma}\label{lem:ref-cones}
Let $\Gamma$ be a $\InfApex$-bounded configuration. There are numbers $L 
\in 
\N$ and $\theta \in  (0,\frac{\pi}{2}]$, and double cones 
$V^1,...,V^L$ centered at $0$ with apex angle 
$\theta$ and symmetry axis $v^1, 
\ldots, v^L \in \S^{d-1}$ such that 
\[
  \forall x\in \R^d \,\,\, \exists m\in \{1, \ldots, L\} \qcolon V^m 
  \subset \Gamma(x) \,.
\]
The constants $L$ and $\theta$ depend on the dimension $d$ and $\InfApex$ but 
not on $\Gamma$ itself.
\end{lemma}

\begin{proof} Obviously 
\[
  \S^{d-1}
  \,\,\subset\,\,
  \bigcup_{v \in \S^{d-1}} V\left(v,\frac{\vartheta}{3}\right)
\]
where $\S^{d-1}$ is the unit sphere in $\TheSpace$.
Since $\S^{d-1}$ is compact and the right hand side is an open 
cover of 
$\S^{d-1}$, one can choose finite many $v^1,...,v^L \in \S^{d-1}$ such that
\[
  \S^{d-1}
  \,\,\subset\,\,
  \bigcup_{m=1}^L V\left(v^m,\frac{\vartheta}{3}\right) .
\]
Define $V^m = V\left(v^m,\frac{\vartheta}{3}\right)$ for $m=1,...,L$. Now the 
claim follows with $\theta= \vartheta/3$. 
\end{proof}
\begin{definition}\label{def:ref-cones}
In the sequel, we write $V^m[x]$ instead of $V^m+x$. We call the set $\{ 
V^m 
| 1\leq i\leq L\}$ the \emph{family of reference cones} associated to $\Gamma$. 
Each element is called a \emph{reference cone}.
Analogous to \autoref{def:start} we set 
\[
  V_r^m \,=\,
    \big\{\, u \in V^m \,\,\big|\,\, \overline{B_r} \subset V^m \,\big\}
  , \,\,\,\,\,
  V_r^m[x]\,=\,V^m_r + x \,.
\]
\end{definition}

With help of \autoref{lem:ref-cones} we can define a new configuration 
that has useful properties. The following corollary is the key tool for 
our reasoning in \autoref{sec:prelude} and \autoref{sec:proof_discrete}. 

\begin{corollary}\label{cor:ref-config}
 Let $\Gamma$ be a $\vartheta$-bounded configuration. Then there exists 
another configuration $\widetilde{\Gamma}$ that fulfills 
$\#\widetilde{\Gamma}(\R^d)<\infty$ and for every $x \in \R^d$
\[ \widetilde{\Gamma}(x) \,\subset\, \Gamma(x). \]
The minimum of apex angles of cones in $\widetilde{\Gamma}(\R^d)$ is 
$\vartheta$. 
\end{corollary}

\begin{proof}
 Let $V^1,...,V^L$ be the double cones from the preceding lemma. Define sets
 \begin{align*}
  M_1&\,\,=\,\,\{\, x \in \R^d \,\,|\,\, V^1 \subset \Gamma(x) \,\}\\
  M_2&\,\,=\,\,\{\, x \in \R^d \,\,|\,\, V^2 \subset \Gamma(x) \,\} \setminus M_1 \\
    & \,\,\,\,\,\vdots\\
  M_L&\,\,=\,\,\{\, x \in \R^d \,\,|\,\, V^L \subset \Gamma(x) \,\} \setminus M_{L-1}.
 \end{align*}
Now 
\[ \R^d \,\,=\,\, \bigcup_{\squash{1\leq i\leq L}} M_i\]
and this union is disjunct. Define $\widetilde{\Gamma}:\R^d \to \mathcal{V}, x 
\mapsto V^i$ for $x \in M_i$ and arrive at the assertion.
\end{proof}

\begin{definition}\label{def:cubes}
For $h >0$ and $u=(u_1,...,u_d) \in \R^d$ let
\begin{align*}
  A_h(u)
  \,\,=\,\,
  \big\{\,
    x \in \R^d
    \,\,\big|\,\,
    \|x-u\|_\infty < h/2
  \,\big\}
\end{align*}
the open cube with center $u$. The half-closed cube with center $u$ will be 
denoted by
\[
  \widetilde{A}_h(u)
  \,\,=\,\,
  \prod_{i=1}^d\Big[u_i-\frac{h}{2},u_i+\frac{h}{2}\Big).
\]
\end{definition}

\begin{remark}
 Half-closed cubes are only needed in one proof in 
\autoref{sec:application_discrete}.
\end{remark}

Let $\Gamma$ be a $\vartheta$-admissible configuration and $\{V^1,...,V^L\}$  
a family of reference cones according to \autoref{lem:ref-cones}. Then
 our assumption \eqref{assum: measurability} implies for any $V \in 
\mathcal{V}$ that the set
\[
  \{\, x \in \R^d \,\, | \,\, V \subset \Gamma(x) \,\}
\]
is Lebesgue measurable.
This implication is due to \cite[Thm. 4.4]{Debreu67}.

Given $h>0, u\in \R^d$ and $m\in \{1,...,L\}$, we set
\[ 
  A_h^m(u)
  \,\,=\,\,
  \{\,
    x\in A_h(u)
    \,\,|\,\,
    V^m \subset \Gamma(x)
  \,\}\,.
\]
An index $m \in \{1,...,L\}$ is called \emph{$h$-favored by majority at $u$ (or 
short: $h$-favored index at $u$)} if 
\[
  \lambda_d(A_h^m(u))
  \,\,=\,\,
  \max\limits_{i \in \{1,...,L\}} \lambda_d(A_h^i(u))\,.
\]
Here $\lambda_d$ is the Lebesgue measure on $\R^d$.
Note that $\lambda_d(A_h^m(u)) \geq L^{-1}\lambda_d(A_h(u))$ for every 
$h$-favored index at 
$u$. This follows directly from
\[
  A_h(u)
  \,\,=\,\,
  \bigcup_{i \in \{1,...,L\}}A_h^i(u) \,.
\]
It is clear that the choice of a $h$-favored index is in general not unique.

Now we state an elementary result for the intersection of cones which will be 
very helpful for us.

\begin{lemma}\label{lem:cone in intersection}
 Let $V$ be a cone with apex angle $\vartheta$ and let $h>0$. Then for each $x 
\in \R^d$ and each $\xi \in A_h(x)$
\[
  V_{h\sqrt{d}}[\xi]
  \,\,\subset\,\, V_{\frac{h}{2}\sqrt{d}}[x]
  \,\,\subset\,\, V[\xi]\,.
\]
In other words
\[
  \bigcup_{\xi \in A_h(x)} V_{h\sqrt{d}}[\xi]
  \,\,\subset\,\,
  V_{\frac{h}{2}\sqrt{d}}[x]
  \,\,\subset\,\,
  \bigcap_{\xi \in A_h(x)}V[\xi]\,.
\]
\end{lemma}

\begin{proof}
Let $\ell>0$. Notice 
\begin{align*}
  \zeta \,\in\,
    \bigcap_{\xi \in \overline{B_\ell}}V[\xi]
    &\,\,\,\Leftrightarrow\,\,\,
      \forall \xi \in \overline{B_\ell} \qcolon \zeta - \xi \in V
    \\
    &\,\,\,\Leftrightarrow\,\,\,
    \zeta - \overline{B_\ell} \,\subset\, V
    \\
    &\,\,\,\Leftrightarrow\,\,\, \overline{B_\ell(\zeta)}  \,\subset\, V
    \\
    &\,\,\,\Leftrightarrow\,\,\, \zeta \in V_\ell.
\end{align*}
This means
\begin{align}\label{stern}
  V_\ell
  \,\,=\,\,
  \bigcap_{\xi \in  \overline{B_\ell}}V[\xi]\,.
\end{align}
On the other hand, for $\zeta \in V_{2 \ell}$, we have 
$\overline{B_\ell(\zeta)}\subset V_\ell$. This is equivalent to
\[
  \forall \zeta \in V_{2\ell} \,\,\, \forall \xi \in \overline{B_\ell} \qcolon
  \zeta + \xi \in V_\ell \,.
\]
In other words
\begin{align}\label{kreuz}
  \bigcup_{\xi \in \overline{B_\ell}}V_{2 \ell}[\xi] 
  \,\,\subset\,\,
  V_\ell \,.
\end{align}
From \eqref{stern} and \eqref{kreuz} we conclude for every $\xi \in 
\overline{B_\ell}$
\[
  V_{2\ell}[\xi]
  \,\,\subset\,\,
  V_\ell
  \,\,\subset\,\,
  V[\xi] \,.
\] 
Translation by $x \in \R^d$ yields 
\[
  V_{2\ell}[\xi]
  \,\,\subset\,\,
  V_\ell[x]
  \,\,\subset\,\,
  V[\xi]  \quad\quad \forall \xi \in \overline{B_\ell(x)}\,.
\]
Now set $\ell = \frac{h}{2}\sqrt{d}$ and observe that $A_h(x)\subset 
\overline{B_\ell(x)}.$
\end{proof}

\section{Application of the discrete problem}\label{sec:application_discrete}

In this section, we show how to derive \autoref{theo:main} 
from \autoref{theo:main_discrete}. The idea is to provide an 
$h\Z^d$-Version of 
\autoref{theo:main_discrete}, applying it to a discrete version of the kernel 
$k$ from \autoref{theo:main} and then pass to the limit $h\to 0$. We also prove 
\autoref{cor:H-are-equal}.

By scaling, we can deduce the following 
$h\Z^d$-Version from \autoref{theo:main_discrete}. 

\begin{corollary}\label{cor:main_discrete} 
Let $\Gamma$ be a $\InfApex$-bounded configuration and let $\TheScale > 0$. 
Let $\omega:\TheScale\TheLattice \times \TheScale\TheLattice \to [0,\infty]$ be 
a 
function 
satisfying $\omega(x,y) = \omega(y,x)$ and 
\begin{align}\label{assum:main_discrete cor}
  \Lambda^{-1}
  \big( 
    \mathbbm{1}_{V^\Gamma [x]}(y) + 
    \mathbbm{1}_{V^\Gamma [y]}(x)
  \big) |x-y|^{-d-\alpha}
  \,\,\leq\,\,
  \omega(x,y)
  \,\,\leq\,\,
  \Lambda |x-y|^{-d-\alpha} 
\end{align}
for $|x-y| >\minDist h$, where $\minDist>0$, $\Lambda \geq 1$ are some 
constants. 
There exist constants $\kappa \geq 1$ and $c > 0$, such 
that for every $R>0$, every $x_0\in \R^d$, and every function  
$f:(B_{\kappa R} \cap h\Z^d) \to \R$, the inequality
\[
  c
  \sum
  \limits_{\underset{\DistOf{x}{y}> \minDist\TheScale}{x,y\in B_R \cap h\Z^d}}
  (f(x)-f(y))^2 |x-y|^{-d-\alpha}
  \,\,\leq\,\, 
  \sum
  \limits_{\underset{\DistOf{x}{y}> \minDist\TheScale}{x,y\in B_{\kappa R} \cap h\Z^d}}
  (f(x)-f(y))^2 \omega(x,y)
\]
holds.
The constant $c$ depends on $\Lambda, \InfApex, R_0$ and on the dimension $d$. 
It does not depend on $\omega, \Gamma$ and $h$.
\end{corollary}

\begin{proof} 
Let:
\[ 
  \begin{aligned}
   M
   &\,\,=\,\,
   \Bigg\{\,
     \omega : h\Z^d \times h\Z^d\to [0,\infty] \,\, \Bigg| \,\,
     \begin{array}{@{}l@{}}
       \omega(x,y)=\omega(y,x) \text{ and } \\
       \eqref{assum:main_discrete cor} \text{\ for some configuration\ }
       \Gamma\text{\ with\ }\vartheta>0
     \end{array}
   \Bigg\}
  \\
  N
  &\,\,=\,\,
  \Bigg\{\,
    \omega:\Z^d \times \Z^d\to [0,\infty]
    \,\,\Bigg|\,\,
    \begin{array}{@{}l@{}}
      \omega(x,y)=\omega(y,x) \text{\ and } \\
      \eqref{assum:main_discrete} \text{\ for some configuration\ }
      \Gamma\text{\ with\ }\vartheta>0
    \end{array}
  \,\Bigg\}
  \end{aligned}
\]
Every element $\omega \in M$ is of the form 
$h^{-d-\alpha}\tilde{\omega}(h^{-1}x,h^{-1}y)$ for some $\tilde{\omega} \in N$. 
If $R>0,x_0\in \R^d$ and $f:B_{\kappa R}(x_0)\cap h\Z^d\to \R$ is some 
function, we define the function
$g:B_{\kappa R}(x_0)\cap \Z^d \to \R$ by $g(x)=f(hx)$. Then with use of 
\autoref{theo:main_discrete}:
\begin{alignat*}{3}
& \quad c &&\sum\limits_{\underset{\DistOf{x}{y}> \minDist\TheScale}{x,y\in 
B_R(x_0) \cap h\Z^d}} &&\, (f(x)-f(y))^2|x-y|^{-d-\alpha}\\
&=  c && \, \, \sum\limits_{\underset{\DistOf{x}{y}>\minDist}{x,y\in B_R(x_0) 
\cap \Z^d}} &&\, (g(x)-g(y))^2h^{-d-\alpha}|x-y|^{-d-\alpha} \\
&\leq &&\, \sum\limits_{\underset{\DistOf{x}{y}>\minDist}{x,y\in B_R(x_0) \cap 
\Z^d}} &&(g(x)-g(y))^2h^{-d-\alpha}\tilde{\omega}(x,y) \\
& = &&\sum\limits_{\underset{\DistOf{x}{y}> \minDist\TheScale}{x,y\in 
B_{\kappa R}(x_0) \cap h\Z^d}} 
&&(f(x)-f(y))^2h^{-d-\alpha}\tilde{\omega}(h^{-1}x,h^{-1}y) \\
&= &&\sum\limits_{\underset{\DistOf{x}{y}> \minDist\TheScale}{x,y\in 
B_{\kappa R}(x_0) \cap h\Z^d}} &&(f(x)-f(y))^2\omega(x,y).
\end{alignat*}
This proves the claim.
\end{proof}

\subsection{The discrete version of the kernel}
In this subsection we will always assume that 
$\Gamma$ is a fixed $\InfApex$-admissible configuration and $\{V^m\}_{1\leq 
m\leq L}$ is the associated family of reference cones.
We will always denote the symmetry axis 
of a reference cone $V^m$ by $v^m$ ($m \in \{1,...,L\}$).

For $k:\R^d \times \R^d \to [0,\infty]$ a nonnegative
measurable function and $h>0$, we define $\omega_h^k: 
h\Z^d \times h\Z^d \to \R$ by
\[
  \omega_h^k (x,y)
  \,\,=\,\,
  h^{-2d}\int\limits_{\Squash{A_h(x) \times A_h(y)}} k(s,t) \dst.
\]
Note that $\omega_h^k(x,y)$ may be infinite for $x$ and $y$ from neighboring 
cubes. 

We want to apply \autoref{cor:main_discrete} to $\omega=\omega_h^k$. 
Therefore, we need to make sure that the function $\omega_h^k$ satisfies 
\eqref{assum:main_discrete cor}.
First, we show this claim for $h=1$.
The next three technical lemmas are 
tailor-made for this purpose.

\begin{lemma}\label{lem:min.dist}
For all $x, y \in \Z^d$, all $1$-favored indices $m$ at 
$x$ and $n$ at $y$, all $t\in A_1^n(y)$, and all $s\in A_1^m(x)$,
the inequality
\[
  \mathbbm{1}_{V^m_{\sqrt{d}/2}[x]}(t) + \mathbbm{1}_{V^n_{\sqrt{d}/2}[y]}(s) 
  \,\,\geq\,\, 
    \mathbbm{1}_{V^m_{\sqrt{d}}[x]}(y) + 
\mathbbm{1}_{V^n_{\sqrt{d}}[y]}(x)
\]
holds.
\end{lemma}

\begin{proof}
Let $x,y \in \Z^d$ and let $m$ be a $1$-favored index at $x$.
Assume $y \in V_{\sqrt{d}}^m[x]$. Then, $B_{\sqrt{d}/2}(y)\subset 
V_{\sqrt{d}/2}^m[x]$. Therefore
\[
  A_1^n(y) \,\,\subset\,\,
  A_1(y) \,\,\subset\,\,
  B_{\sqrt{d}/2}(y) \,\,\subset\,\,
  V_{\sqrt{d}/2}^m[x].
\] 
\end{proof}

The assertion of the following 
lemma is obviously true.

\begin{lemma}\label{lem:new config}
Let $r>0$. There is an apex angle $\theta>0$ such that for every $m \in 
\{1,...,L\}$ there exists an axis $v(m) \in \R^d$ so that
\[
  \left(\, V(v(m),\theta) \cap \Z^d \,\right)
  \,\,\subset\,\,
  \left(\, V^m_r \cap \Z^d \,\right) .
\]
\end{lemma}

\begin{lemma}\label{lem:cubes}
 For every $h>0$, all $x,y \in \Z^d$ with $|x - y|>\sqrt{d}h$ and all
 $s \in A_h(x), t \in A_h(y)$, the following holds:
    \[
        \frac{1}{2\sqrt{d}} |x-y| \,\,<\,\, |s-t| \,\,<\,\, 2\sqrt{d} |x-y|.
    \]
\end{lemma}
\begin{proof}
  This is about comparing the euclidean norm to the maximum norm on $\R^d$.
  Note that for any vector $v\in\R^d$, we have:
  \begin{equation}\label{eq:norms-on-R-n}
    |v|_\infty \,\,\leq\,\, |v| \,\,\leq\,\, \sqrt{d} |v|_\infty.
  \end{equation}
  Let $h=1$ and $x,y \in \Z^d$ with $|x-y|> \sqrt{d}$.
  Since the maximum norm takes only integer values on lattice points 
  and $|x-y|>\sqrt{d}$, it
  follows that $|x-y|_\infty \geq 2$. As a consequence of the triangle 
  inequality, we record for $s \in A_1(x)$ and $t\in A_1(y)$:
  \[
    \frac12 |x-y|_\infty
    \,\,\leq\,\, |x-y|_\infty -1 
    \,\,<\,\, |s-t|_\infty< |x-y|_\infty+1
    \,\,\leq\,\, 2|x-y|_\infty.
  \]
  Using~\eqref{eq:norms-on-R-n}, we conclude:
  \begin{align*}
    |s-t| &\,\,\leq\,\, \sqrt{d}|s-t|_\infty \,\,<\,\,
            2\sqrt{d}|x-y|_\infty \,\,\leq\,\, 2\sqrt{d} |x-y|,\\
    |x-y| &\,\,\leq\,\, \sqrt{d}|x-y|_\infty \,\,<
            2\sqrt{d}|s-t|_\infty \,\,\leq\,\, 2\sqrt{d} |s-t|.
  \end{align*}
  The general case for arbitrary $h>0$ follows by scaling.
\end{proof}
\begin{proposition}\label{prop:test-fct} 
Let $k:\R^d \times \R^d 
\to [0,\infty]$ be a symmetric and measurable function 
satisfying 
\eqref{assum:main} for a $\InfApex$-admissible configuration $\Gamma$. Then 
there are constants $C=C(d,\vartheta)>0$ and $\vartheta' \in  
(0,\frac{\pi}{2}]$ 
and a 
$\vartheta'$-bounded configuration 
$\Gamma'$ such that
for all $x,y \in \Z^d$ with $|x-y|>\sqrt{d}$:
\begin{align*}
  C\Lambda^{-1} \left(\mathbbm{1}_{V^{\Gamma'} [x]}(y) + 
  \mathbbm{1}_{V^{\Gamma'} [y]}(x)\right) |x-y|^{-d-\alpha}
  \,\,\leq\,\,
  \omega^k_1 (x,y)
\end{align*}
The angle $\vartheta'$ does only depend on $\theta$ and on the infimum 
$\InfApex$ of 
the apex angles of all cones in $\Gamma$. There is no further dependence on 
$\Gamma$. 
\end{proposition}

\begin{proof}
Note:
\[
  \omega^k_1(x,y)
  \,\,\geq\,\,
  \Lambda^{-1} \int_{A_1(x) \times A_1(y)}
  \left[ 
    \mathbbm{1}_{V^{\Gamma}(s)}(t) + \mathbbm{1}_{V^{\Gamma}(t)}(s)
  \right] 
  |t-s|^{-d-\alpha} \dst.
\]
Therefore, we just need to concentrate on the integral. Let $m$ be a 
$1$-favored index at $x$ and $n$ be a $1$-favored index at $y$. Then with use
of \autoref{lem:cone in intersection}, \autoref{lem:min.dist}, \autoref{lem:new 
config} and \autoref{lem:cubes}, we estimate
\begin{align*}
 &\int_{A_1(x) \times A_1(y)} \left[ \mathbbm{1}_{V^{\Gamma}[s]}(t) + 
\mathbbm{1}_{V^{\Gamma}[t]}(s) \right] |t-s|^{-d-\alpha} \dst \\
 \geq\,\, &\int_{A_1^m(x) \times A_1^n(y)} \left[ \mathbbm{1}_{V^m[s]}(t) + 
\mathbbm{1}_{V^n[t]}(s) \right] |t-s|^{-d-\alpha} \dst \\
 \geq\,\, &\int_{A_1^m(x) \times A_1^n(y)}\left[ 
\mathbbm{1}_{V_{\sqrt{d}/2}^m[x]}(t) + 
\mathbbm{1}_{V_{\sqrt{d}/2}^n[y]}(s) \right] |t-s|^{-d-\alpha} \dst \\
 \geq\,\, & \frac{1}{(2\sqrt{d})^{d+\alpha}} \lambda_{d\times d}(A_1^m(x) \times 
A_1^n(y)) \left[  
\mathbbm{1}_{V_{\sqrt{d}}^m[x]}(y) + \mathbbm{1}_{V_{\sqrt{d}}^n[y]}(x) \right] 
|x-y|^{-d-\alpha} \\
 \geq\,\, & \frac{1}{(2\sqrt{d})^{d+\alpha}} \lambda_{d \times d}(A_1^m(x) \times 
A_1^n(y)) \left[  
\mathbbm{1}_{V(v(m), \theta)[x] }(y) + \mathbbm{1}_{V(v(n),\theta)[y]}(x) 
\right] |x-y|^{-d-\alpha}.
\end{align*}
Now the claim follows with $C\,=\,\frac{1}{(2\sqrt{d})^{d+2} \cdot L^2} \,\leq\, 
\frac{1}{(2\sqrt{d})^{d+\alpha}}\lambda_{d\times d}(A_1^m(x) 
\times A_1^n(y))$ and some appropriate choice of $\Gamma'$. 
\end{proof}

\begin{corollary}\label{prop:important} 
Let $k:\R^d \times 
\R^d \to [0,\infty]$ be a symmetric and measurable function satisfying 
\eqref{assum:main} for a $\InfApex$-admissible configuration $\Gamma$. Then 
there are $\vartheta' >0$ and $C>0$ so that for each $h 
>0$  
there is a configuration $\Gamma^h$ on $\R^d$ with the following properties:
\begin{enumerate}[label=\roman*)]
 \item [(i)] The infimum of the apex angles of all cones in $\Gamma^h(\R^d)$ 
equals $\vartheta'$.
 \item [(ii)] For all $x,y \in h\Z^d$ with $|x-y|> \sqrt{d}h$, the 
inequalities
 \begin{align}\label{eq:main-problem}
   C^{-1} \left(\mathbbm{1}_{V^{\Gamma^h} [x]}(y) + 
   \mathbbm{1}_{V^{\Gamma^h} [y]}(x)\right) |x-y|^{-d-\alpha}
   \,\,\leq\,\,
   \omega^k_h(x,y)
   \,\,\leq\,\,
   C |x-y|^{-d-\alpha}
  \end{align}
  hold.
\end{enumerate}
\end{corollary}

\begin{proof}
For $h>0$ define a new configuration $\Gamma_h$ on $\R^d$ by $ \Gamma_h(x) = 
\Gamma(hx)$. Note that the infimum of the apex angles of all cones in 
$\Gamma_h(\R^d)$ 
is the same as the infimum of the apex angles of all cones in $\Gamma(\R^d)$. 
It 
does 
not depend on $h$. Note also, that \eqref{assum: measurability} holds true for 
$\Gamma$ if and only if \eqref{assum: measurability} holds true for $\Gamma_h$. 
Therefore, $\Gamma_h$ is a $\InfApex$-admissible configuration.\newline
Define $k_h : \R^d \times \R^d \to [0, \infty]$ via
$k_h(x,y)=k(hx,hy)h^{d+\alpha}$. Since $k$ satisfies \eqref{assum:main} we also 
have for all $x,y \in \R^d$:
\begin{align}\label{1}
  \Lambda^{-1} \left( \mathbbm{1}_{V^\Gamma[hx]}(hy) + 
  \mathbbm{1}_{V^\Gamma[hy]}(hx) \right) |x-y|^{-d-\alpha}
  \,\,\leq\,\,
  k(hx,hy)h^{d+\alpha}
  \,\,\leq\,\,
  \Lambda |x-y|^{-d-\alpha}.
\end{align}
Fix some $h>0$. We note that for all $x,y \in \R^d$ the assertion $hy \in 
V^\Gamma[hx]$ is equivalent to $y \in V^{\Gamma_h}(x)$. This together with 
\eqref{1} shows that $k_h$ and $\Gamma_h$ satisfy \eqref{assum:main}. 
Therefore, 
we can apply \autoref{prop:test-fct} to $\Gamma=\Gamma_h$ and $k=k_h$. We 
obtain 
a configuration $(\Gamma_h)'$ with a positive infimum of the apex angles of all 
cones  $\vartheta'$ and some constant $C>0$ such that for all $x,y \in \Z^d$ 
with $|x-y|> \sqrt{d}$, we have
\begin{align}\label{2}
  C\Lambda^{-1}
  \left( 
    \mathbbm{1}_{V^{(\Gamma_h)'}[x]}(y)+\mathbbm{1}_{V^{(\Gamma_h)'}[y]}(x)
  \right) |x-y|^{-d-\alpha}
  \,\,\leq\,\,
  \omega_1^{k_h}(x,y).
\end{align}
Note that $\vartheta'$ does only depend on the infimum of the apex angles of 
all cones in $\Gamma$. We define a new configuration $(\Gamma_h)'_{h^{-1}}$ via 
$(\Gamma_h)'_{h^{-1}}(x) = (\Gamma_h)'(h^{-1}x)$. The infimum of the apex 
angles 
of all cones in this new configuration is obviously still $\vartheta'$. Since 
for all $x, y \in \Z^d$ 
\[
  y \in V^{(\Gamma_h)'}[x]
  \,\,\,\Leftrightarrow\,\,\,
  hy \in V^{(\Gamma_h)'_{h^{-1}}}[hx]\,, 
\]
inequality \eqref{2} is equivalent to 
\begin{align}\label{7}
  C\Lambda^{-1}
  \left( 
    \mathbbm{1}_{V^{(\Gamma_h)'_{h^{-1}}}[hx]}(hy)+\mathbbm{1}_{V^{(\Gamma_h)'_{h^{-1}}}[hy]}(hx)
  \right) |x-y|^{-d-\alpha}
  \,\,\leq\,\,
  \omega_1^{k_h}(x,y)
\end{align}
for all $x, y \in \Z^d$ with $|x-y|> \sqrt{d}$. \\
Now let $x, y \in h\Z^d$ with $|x-y|> \sqrt{d}h$. Then $h^{-1}x,h^{-1}y \in 
\Z^d$ 
with 
$|h^{-1}x-h^{-1}y|> \sqrt{d}$. With use of $\eqref{7}$ and the transformation 
formula 
for integrals we obtain
\begin{align*}
 &
  C\Lambda^{-1}
  \left( 
    \mathbbm{1}_{V^{(\Gamma_h)'_{h^{-1}}}[x]}(y)+\mathbbm{1}_{V^{(\Gamma_h)'_{h^{-1}}}[y]}(x)
  \right) |x-y|^{-d-\alpha} \\
 =\,\,&
  C \Lambda^{-1}
  \left( 
    \mathbbm{1}_{V^{(\Gamma_h)'}[h^{-1}x]}(h^{-1}y)+\mathbbm{1}_{V^{(\Gamma_h)'}[h^{-1}y]}(h^{-1}x)
  \right) |h^{-1}x-h^{-1}y|^{-d-\alpha} h^{-d-\alpha}\\
 \leq\,\,&
   \omega_1^{k_h}(h^{-1}x,h^{-1}y)h^{-d-\alpha}\\
 =\,\,&
   \omega_h^k(x,y).
\end{align*}
The upper bound in (ii) is just a consequence of \autoref{lem:cubes}.\newline
The claim follows with $\Gamma^h=\left(\Gamma_h\right)'_{h^{-1}}$. 
\end{proof}

\subsection{Proof of the continuous version} In this part we prove 
\autoref{theo:main} and \autoref{cor:H-are-equal}.

\begin{proof}[Proof of~\autoref{theo:main}]
  The inequality~\eqref{eq:main-result} is obviously true if the right
  hand side is infinite. Hence, we can restrict ourselves to functions
  $f\in H_k(B)$. The following \autoref{lem:H-are-equal-on-balls}
  provides comparability of the seminorms
  $|\cdot|_{H_k(B)}$ and $|\cdot|_{H^{\frac{\alpha}{2}}(B)}$, which
  implies \eqref{eq:main-result}. The proof is complete.
\end{proof}

\begin{lemma}\label{lem:H-are-equal-on-balls}
 Let $B\subset \R^d$ be a ball. Let $\alpha\in (0,2)$ and
$\Gamma$ be a $\vartheta$-admissible configuration. Let \( k:\R^d \times \R^d 
\to [0,\infty] \) be a measurable function satisfying $k(x,y) = k(y,x)$ and 
\eqref{assum:main}.
The spaces $H_k(B)$ and $H^{\frac{\alpha}{2}}(B)$ coincide. Furthermore 
\[
  |\cdot|_{H^{\frac{\alpha}{2}}(B)}\,\,\leq\,\, c |\cdot|_{H_k(B)}
  \,\,\,\text{on}\,\,\,
  H_k(B)\,=\,H^{\frac{\alpha}{2}}(B)
\]
for a constant $c\geq 1$ independent of the ball $B$. For $0<\alpha_0\leq 
\alpha <2$ the constant depends on $\alpha_0$ but not on $\alpha$.
\end{lemma}

\begin{proof} In view 
of~\eqref{eq:one-direction}, we only have to show the inclusion 
$H_k(\TheBall) \subset H^{\frac{\alpha}{2}}(\TheBall)$. Note that comparability 
of the 
seminorms  implies comparability of norms. Hence, we shall show
  \begin{align}\label{eq:key}
    |f|_{H_k(\TheBall)} \,\,\asymp\,\, |f|_{H^{\frac{\alpha}{2}}(\TheBall)} 
    \quad\text{for all\ } f \in H_k(B).
  \end{align}
  Let $R>0$, $x_0\in \R^d$ and $\kappa$ as in \autoref{cor:main_discrete}. In 
the sequel we use the notation $B=B_R(x_0)$ and $B^\ast = B_{\kappa R}(x_0)$.
Let $f \in H_k(B^\ast)$. For $h\in (0,1)$ we consider the following piecewise 
constant approximation of $f$. We define for $x \in h\Z^d\cap B^\ast$ 
\[
  f_h(x) \,\,=\,\,
  h^{-d}\int_{A_{h}(x)\cap B^\ast} f(s)\, \d s\,.
\]
Because of \autoref{prop:important}, 
  there is a constant $C>0$ and a configuration $\Gamma_h$ with $\vartheta'>0$ 
such that for all 
  $x,y\in 
  h\Z^d$ with $|x-y|> \sqrt{d}h$ the inequalities
  \[
    C^{-1} \left(\mathbbm{1}_{V^{\Gamma_h} [x]}(y) + 
    \mathbbm{1}_{V^{\Gamma_h} [y]}(x)\right) |x-y|^{-d-\alpha}
    \,\,\leq\,\,
    \omega_h^k(x,y)
    \,\,\leq\,\,
    C|x-y|^{-d-\alpha}
  \]
  hold.
  Thus $\omega=\omega_h^k$ together with $\Gamma = \Gamma_h$ fulfill 
  \eqref{assum:main_discrete cor} for  $R_0=\sqrt{d}$ and $\Lambda=C$. 
  \autoref{cor:main_discrete} implies the existence of $c>0$, independent of 
$f,R,\alpha$ and $h$, so that 
  \begin{align*} 
  c \sum\limits_{{\underset{\DistOf{x}{y}> \sqrt{d} \TheScale}{x,y\in B \cap 
  h\Z^d}}} 
  (f_h(x)-f_h(y))^2 |x-y|^{-d-\alpha}
  \,\,\leq\,\,
  \sum\limits_{\underset{\DistOf{x}{y}>
  \sqrt{d}\TheScale}{x,y\in B^\ast \cap h\Z^d}}(f_h(x)-f_h(y))^2 
  \omega(x,y).
  \end{align*}
  Using \autoref{lem:cubes} we obtain
  \begin{align}\label{discret}
  c &\sum\limits_{{\underset{\DistOf{x}{y}> \sqrt{d}\TheScale}{x,y\in B 
\cap 
  h\Z^d}}}(f_h(x)-f_h(y))^2\int_{A_h(x)\times A_h(y)}|s-t|^{-d-\alpha}\, 
\d(s,t) \nonumber \\
  \leq\,\, &\sum\limits_{\underset{\DistOf{x}{y}> \sqrt{d}\TheScale}{x,y\in 
B^\ast 
  \cap h\Z^d}} (f_h(x)-f_h(y))^2 \int_{A_h(x)\times 
  A_h(y)}k(s,t) \, \mathrm{d}(s,t), 
  \end{align}
  for a constant $c>0$ that differs from the one above by a factor only 
depending on the dimension $d$. 

For technical reasons, we  need the property that every $x$ in $\R^d$ is 
contained in some cube. Therefore we consider 
half-closed cubes. Given $x=(x_1,...,x_d) \in \R^d$ and 
$h\in (0,1)$, we use the notation
\[ \widetilde{A}_h(x) \,\,=\,\, 
\prod_{i=1}^d\Big[x_i-\frac{h}{2},x_i+\frac{h}{2}\Big).\]
For $h\in (0,1)$, we define a function $g_h: \R^d \times \R^d \to \R$ via
\[
  g_h(s,t) \,\,=\,\,\sum_{x,y \in 
h\Z^d} 
  \left[(f_{h}(x)-f_{h}(y))^2 k(s,t)\, \mathbbm{1}_{\widetilde{A}_{h}(x)\times 
\widetilde{A}_{h}(y)}(s,t) \,
\mathbbm{1}_{\{x,y\in B^\ast |  \sqrt{d}h < |x-y|\}}(x,y)\right]
\]
  and claim that $g_h$ converges for $h\to 0$ almost everywhere
  to the function $g:\R^d \times \R^d \to \R$ with
\[   g(s,t) \,\,=\,\, (f(s)- f(t))^2k(s,t)\mathbbm{1}_{B^\ast \times B^\ast}(s,t).\]
  Indeed, $g_h(s,t)=(f_{h}(x_h)-f_{h}(y_h))^2k(s,t)$ for appropriate 
points $x_h$ and $y_h$. We conclude with help of \autoref{lem:LebesgueDiff} 
$g_h(s,t)\to g(s,t)$ for almost 
every $(s,t)\in B^\ast \times B^\ast$.
In the same way we can show that the function $\widetilde{g}_h:\R^d \times \R^d 
\to \R$ with  
  \begin{align*}
  \widetilde{g}_h(s,t) \,\,=\,\,\sum_{x,y \in 
h\Z^d} &\Big[
  (f_{h}(x)-f_{h}(y))^2 
|s-t|^{-d-\alpha} \, \mathbbm{1}_{\widetilde{A}_{h}(x)\times 
\widetilde{A}_{h}(y)}(s,t)\\ &\times 
\mathbbm{1}_{\{x,y\in B | \sqrt{d}h_n< |x-y|\}}(x,y)\Big]
  \end{align*}
 converges for $h\to 0$ pointwise a.e. to 
   \begin{align*}
   &\widetilde{g}:\R^d \times \R^d \to \R, \\
   &\widetilde{g}(s,t) \,\,=\,\, (f(s)- 
f(t))^2|s-t|^{-d-\alpha}\mathbbm{1}_{B\times B}(s,t).
  \end{align*}

  For the right hand side in \eqref{discret} this implies with help of 
dominated convergence
  \begin{align*}
  \sum\limits_{{\underset{\DistOf{x}{y}> \sqrt{d}\TheScale}{x,y\in B^\ast
  \cap h\Z^d}}}  & (f_h(x)-f_h(y))^2 \int_{\widetilde{A}_h(x)\times 
  \widetilde{A}_h(y)}k(s,t) \, \mathrm{d}(s,t) \\
   = &\int_{\R^d \times \R^d} g_h(s,t) \dst
   \,\,\,
  \overset{h \to 0 }{\longrightarrow}
  \,\,\,
  \int_{\R^d \times \R^d} 
  g(s,t) \dst  
  \end{align*}
  With regard to the left hand side of 
\eqref{discret}, note that the Fatou lemma implies  
    \begin{align*}
  \liminf\limits_{h \to 0} & \sum\limits_{{\underset{\DistOf{x}{y}> 
\sqrt{d}\TheScale}{x,y\in B\cap h\Z^d}}} (f_h(x)-f_h(y))^2 
\int_{\widetilde{A}_h(x)\times 
  \widetilde{A}_h(y)}|s-t|^{-d-\alpha} \, \mathrm{d}(s,t) \\
  =\,\, & \liminf_{h \to 0}\int_{\R^d \times \R^d} \widetilde{g}_h(s,t) \, \d(s,t)
  \,\,\geq\,\,\int_{\R^d\times \R^d}\widetilde{g}(s,t)\d (s,t).
  \end{align*}
  In conclusion, we have shown that the discrete inequality \eqref{discret} 
yields the continuous version
  \begin{align*}
   c |f|_{H^{\frac{\alpha}{2}}(\TheBall)} \,\,\leq\,\, |f|_{H_k(\TheBall^\ast)} 
\quad\text{for all }f\in H_k(\TheBall^\ast).
  \end{align*}
  This is true for every ball $B$, since $c$ is independent of $B$. 
  Therefore, using \autoref{lem:new6.9}, we conclude for each ball $B\subset 
\R^d$ and each $f\in H_k(\TheBall^\ast)$
  \begin{align*}
  c^\ast |f|_{H^{\frac{\alpha}{2}}(\TheBall)} \,\,\leq\,\, |f|_{H_k(\TheBall)} 
  \end{align*}
  for some $c^\ast>0$, independent of the ball $B$. 
  This proves comparability of the seminorms and $H_k(B)\subset 
H^{\tfrac{\alpha}{2}}(B)$.
\end{proof}

\begin{proof}[Proof of \autoref{cor:H-are-equal}]
The comparability constant $c$ in \autoref{lem:H-are-equal-on-balls} is 
independent of the radius $R$ of the 
respective ball. Thus the result for the whole space is obtained in the 
limit $R\to \infty$ using monotone convergence.

Now let $\Omega$ be a bounded Lipschitz domain. In view of 
\autoref{lem:H-are-equal-on-balls} and \autoref{lem:new6.9} we conclude
\[ |\cdot |_{H_k(\Omega)} \,\,\leq\,\, c |\cdot |_{H^{\frac{\alpha}{2}}(\Omega)} \text{ 
on } H_k(\Omega)\]
for a constant $c\geq 1$,
which leads to $H_k(\Omega)\subset H^{\frac{\alpha}{2}}(\Omega)$. Since the 
inclusion $H^{\frac{\alpha}{2}}(\Omega)\subset H_k(\Omega)$ is obvious by the 
definition of $k$, one obtains $H_k(\Omega)=H^{\frac{\alpha}{2}}(\Omega)$.
For the assertions concerning density of smooth functions, note 
that $C^\infty(\overline{\Omega})$ is a dense 
subset of $H^\frac{\alpha}{2}(\Omega)$, cf. \cite[Proposition 4.52]{DeDe12}. 
Furthermore $C_c^\infty(\R^d)$ is a dense subset of $H^\frac{\alpha}{2}(\R^d)$, 
cf. \cite[Proposition 4.27]{DeDe12}.
\end{proof}

\section{A continuous prelude}\label{sec:prelude}
The main tool for the proof of \autoref{theo:main_discrete} is the 
construction of paths connecting two arbitrary points in $\TheLattice$. In 
this section we show the existence of paths connecting two arbitrary points in 
$\TheSpace$. This result is not needed to prove \autoref{theo:main_discrete}. 
Therefore, the reader may skip this section. However, the 
procedure in the continuous setting is similar to the discrete setting, but 
less 
technical. For this reason, reading this section first might provide useful
intuition.
  
  From a configuration
  \(
    \TheConf\mapcolon\TheSpace\rightarrow\SetDoubleCones
  \)
  we construct a directed graph $\DirGraph$ as follows: the vertex
  set is $\TheSpace$ and there is a directed edge from $\ThePoint$
  to $\AltPoint$ if $\AltPoint\in\BasedCone{\ThePoint}$. Note that there are no 
loops in $\DirGraph$
  as $\TheConfOf{\ThePoint}$ is open and does not contain the tip.
  
  We shall be concerned with the question whether $\DirGraph$ is
  connected as an \emph{undirected} graph if the underlying configuration is 
$\TheApex$-bounded. In this case, \autoref{cor:ref-config} allows 
us to assume 
WLOG that the image of $\TheConf$ contains only a finite number of elements. 
Thus, crucial parts of the
  argument can be proved by induction on the number of cones in 
$\TheConf(\TheSpace)$.
  As it often happens, one needs to strike the right balance and the
  statement suitable for induction is a little bit stronger (and
  more technical) than the primary target. We are led to consider
  subgraphs $\RestrGraph$ defined by open subsets
  $\TheOpenSet\subset\TheSpace$ as follows: the vertex set of
  $\RestrGraph$ is still $\TheSpace$ and the rule for oriented
  edges is the same, however, we only put in the edges issuing
  from vertices in $\TheOpenSet$. Note that vertices outside $\TheOpenSet$
  still can be used in edge paths since we are interested in
  undirected connectivity.
  
  In this section we always assume that the configuration $\TheConf$ is 
$\TheApex$-bounded.
  
  The main result of this part is:
  \begin{theorem}\label{cont.connectivity}
    For any connected open set $\TheOpenSet\subset\TheSpace$, any two points
    $\ThePoint,\AltPoint\in\TheOpenSet$ are vertices in the same connected
    component of $\RestrGraph$.
  \end{theorem}
  For the proof of \autoref{cont.connectivity} we need some auxiliary results.
  \begin{definition}
   A point $\ThePoint \in \TheSpace$ is \emph{of type} $\TheDoubleCone$ if 
$\TheDoubleCone=\TheConfOf{\ThePoint}$. Two points $\ThePoint, \AltPoint \in 
\TheSpace$ \emph{have the same type} if 
$\TheConfOf{\ThePoint}=\TheConfOf{\AltPoint}$. 
  \end{definition}

  \begin{lemma}\label{cont.connect-two-of-same-type}
    If two points $\ThePoint,\AltPoint\in\TheOpenSet$ have the same type, then 
there is
    an edge path in $\RestrGraph$ of length at most two connecting them.
  \end{lemma}
  \begin{proof}
    Let $\TheDoubleCone=\TheConfOf{\ThePoint}=\TheConfOf{\AltPoint}$.
    Then the translated double cones $\TheDoubleCone[\ThePoint]$ and
    $\TheDoubleCone[\AltPoint]$ intersect. We pick a point of
    intersection (it may lie outside of $\TheOpenSet$). It has an edge
    incoming from $\ThePoint$ and another edge incoming
    from $\AltPoint$. These two edges form the desired edge path.
  \end{proof}

  \begin{definition} We call $\ThePoint$ \notion{well-connected in 
$\TheOpenSet$} if there
  is an open neighborhood $\TheOpenNbhd$
  of $\ThePoint$ that, considered as a set of vertices in $\RestrGraph$,
  lies entirely in a single connected component of $\RestrGraph$. I.e., the 
point
  $\ThePoint$ is connected by edge paths in $\RestrGraph$ to all points
  of an open neighborhood.
  \end{definition}

The following lemma lists inter alia some important features of well connected 
points.

\begin{lemma}\label{lem:observations}
The following hold.
\begin{enumerate} 
 \item \label{cont.well-connected}  
    For $\AltPoint\in\TheOpenSet$, any point
    $\ThePoint\in\TheOpenSet\intersect\BasedCone{\AltPoint}$
    is well-connected in $\TheOpenSet$.
 \item \label{cont.grow-open-set} 
    If $\TheSmallOpenSet\subset\TheOpenSet$ is an inclusion of open sets, then
    any point $\ThePoint\in\TheSmallOpenSet$ that is well-connected in 
$\TheSmallOpenSet$
    is also well-connected in $\TheOpenSet$.
 \item \label{cont.well-connected-are-dense}
    Any non-empty open set $\TheOpenSet$ contains a point
    that is well-connected in $\TheOpenSet$. In fact, the well-connected
    points are dense in $\TheOpenSet$.
\end{enumerate}
\end{lemma}

\begin{proof}
  For (\ref{cont.well-connected}) we may choose
    $\TheOpenSet\intersect\BasedCone{\AltPoint}$ as the
    open neighborhood. Any two points therein are connected via an
    edge path of length two with $\AltPoint$ as the middle vertex. Therefore 
(\ref{cont.well-connected}) follows.

     Enlarging the open set $\TheSmallOpenSet$ only adds edges to the graph.
    Hence connectivity can only improve. This proves (\ref{cont.grow-open-set}).

    For the proof of (\ref{cont.well-connected-are-dense}) notice that
    existence of a well-connected point follows from 
(\ref{cont.well-connected}).
    Applying the existence statement to smaller open sets
    $\TheSmallOpenSet\subset\TheOpenSet$, density follows in view of
    (\ref{cont.grow-open-set}).
\end{proof}

\begin{lemma}\label{cont.ueber-bande}
    Consider two points $\ThePoint,\AltPoint\in\TheOpenSet$ and let
    $\TheDoubleCone=\TheConfOf{\AltPoint}$ be the cone type of $\AltPoint$.
    Assume that the translated double cone $\TheDoubleCone[\ThePoint]$
    contains a point $\ThrPoint$ of type $\TheDoubleCone$. Then $\ThePoint$
    and $\AltPoint$ are connected.
  \end{lemma}
  Note that we do not assume that
  $\TheDoubleCone=\TheConfOf{\ThePoint}$.
  One may also note that in the situation of the lemma, the point $\ThePoint$
  is well-connected in $\TheOpenSet$.

  \begin{proof}
    Since $\AltPoint$ and $\ThrPoint$ have the same type, they are connected
    by an edge path of length at most two. Now,
    \(
      \ThrPoint\in\TheDoubleCone[\ThePoint]
    \)
    implies
    \(
      \ThePoint\in\TheDoubleCone[\ThrPoint]=\BasedCone{\ThrPoint}
    \).
    Hence, there is an edge from $\ThrPoint$ to $\ThePoint$.
  \end{proof}

  \begin{proof}[Proof of~\autoref{cont.connectivity}]
    According to \autoref{lem:ref-cones}, we may assume that the 
image of $\TheConf$ has at most $L$ different elements since $\TheConf$ is 
$\TheApex$-bounded.
    Therefore, we can use induction on the number 
$\CardOf{\TheConfOf{\TheOpenSet}}$ 
of cones
    realized in $\TheOpenSet$. If there is only a single cone type throughout
    $\TheOpenSet$, any two points $\ThePoint,\AltPoint\in\TheOpenSet$ are
    connected in $\RestrGraph$ by an edge path of length at most two.
    This settles the base of the induction.

    For $\CardOf{\TheConfOf{\TheOpenSet}}>1$, we start with the following
    observation:
    \begin{quote}
      There is a constant $\TheFactor>0$ depending only on the minimum
      apex angle $\TheApex$ such that for any double cone
      $\TheDoubleCone\in\SetDoubleCones$
      and any two points $\ThePoint,\AltPoint\in\TheSpace$ of distance
      $\NormOf{\ThePoint-\AltPoint}<\TheFactor$, the intersection
      \(
        \TheDoubleCone[\ThePoint]
        \intersect
        \TheDoubleCone[\AltPoint]
      \)
      contains a point in $\BallOf[1]{\ThePoint}$.      
    \end{quote}
    Now assume that $\ThePoint$ is well-connected in $\TheOpenSet$ and
    that the $\TheRadius$-ball $\BallOf[\TheRadius]{\ThePoint}$ lies
    entirely in $\TheOpenSet$. We claim that $\ThePoint$ is connected
    to any point $\AltPoint\in\BallOf[\TheFactor\TheRadius]{\ThePoint}$. 
Indeed, consider the cone type $\TheDoubleCone=\TheConfOf{\AltPoint}$
    of $\AltPoint$. If $\TheDoubleCone[\ThePoint]$ contains a point of
    type $\TheDoubleCone$, the points $\ThePoint$ and $\AltPoint$ are
    connected by \autoref{cont.ueber-bande}.

    Otherwise, within the open set
    \(
      \TheSmallOpenSet=\TheOpenSet\intersect\TheDoubleCone[\ThePoint]
      \neq \emptyset
    \)
    the cone type $\TheDoubleCone$ is not realized. We infer by induction
    that all points in $\TheSmallOpenSet$ are mutually connected in
    $\SmallRestrGraph$ and hence in $\RestrGraph$. However,
    $\TheDoubleCone[\AltPoint]=\BasedCone{\AltPoint}$ intersects
    \(
      \TheSmallOpenSet
      \supset
      \BallOf[\TheRadius]{\ThePoint}\intersect\TheDoubleCone[\ThePoint]
    \)
    by the opening observation. Hence $\AltPoint$ is
    connected to a point in $\TheSmallOpenSet$ and therefore to any point
    in $\TheSmallOpenSet$, which contains points arbitrarily close to
    $\ThePoint$. Since $\ThePoint$ is well-connected in $\TheOpenSet$,
    the points $\AltPoint$ and $\ThePoint$ are connected in
    $\RestrGraph$.

    It follows that a well-connected point $\ThePoint\in\TheOpenSet$
    whose $\TheRadius$-neighborhood lies in $\TheOpenSet$ is actually
    connected to any point in its $\TheFactor\TheRadius$-neighborhood.
    Now, density of well-connected points in $\TheOpenSet$ (cf. 
\autoref{lem:observations} (\ref{cont.well-connected-are-dense})) implies that
    $\TheOpenSet$ is covered by overlapping open well-connected subsets.
  \end{proof}

\section{Chaining and renormalization}\label{sec:proof_discrete} In 
this section we provide the chaining argument that leads to the proof of 
\autoref{theo:main_discrete} in \autoref{sec:proof_discrete}. The main result 
of this section is \autoref{theo:path-props}.

Every configuration
  \(
    \TheConf\mapcolon\TheSpace\rightarrow\SetDoubleCones 
  \)
 induces naturally a mapping $\TheConf|_{\TheLattice}$ which we again call 
configuration and denote by $\TheConf$.
 As in the continuous case, a configuration $\TheConf$ defines a directed graph 
$\DirGraph=\DirGraph(\TheConf)$ where the vertex set is given by $\TheLattice$ 
and there is an oriented edge from $\ThePoint$ to $\AltPoint$ if $y \in 
\BasedCone{\ThePoint}$.
   
  We shall be concerned with the question whether $\DirGraph$ is
  connected as an \emph{undirected} graph if $\TheConf$ is 
$\TheApex$-bounded. Therefore, throughout this section we assume without 
further notice that $\TheConf$ is a $\TheApex$-bounded configuration. In 
addition to the continuous version, however, we also want to keep track
  of how far such an edge path might take us away from the
  end points in question.
  \subsection{Auxiliary results}
  A technicality is that we always
  have to use lattice points. Since any closed ball of radius
  $\frac{\SqrtOf{\TheDim}}{2}$ contains a lattice point, we have:
  
  \begin{lemma}\label{lem:discrObservations}
  Let $\TheCone$ be a cone of apex angle at least $\TheApex$. 
  \begin{enumerate}
   \item\label{discr.ensure-lattice-point}
	  Fix $\SmallRadius > 0$ and assume
	  \(
	   \LargeRadius > \frac{\SmallRadius+\SqrtOf{\TheDim}}{\SinOf{\TheApex}}
	  \).
	  Then
	  \(
	    \BallOf[\LargeRadius]{\ThePoint}
	    \intersect
	    \TheCone[\ThePoint]
	  \)
	  contains a lattice point $\AltPoint\in\TheLattice$ with
	  $\BallOf[\SmallRadius]{\AltPoint}\subset \TheCone[\ThePoint]$.
    \item\label{discr.lattice-point-in-intersection}
	  Let $\ThePoint,\AltPoint\in\TheSpace$.
	  Fix $\SmallRadius > \NormOf{\ThePoint-\AltPoint}$ and
	  \(
	    \LargeRadius > 
\frac{\SmallRadius+\SqrtOf{\TheDim}}{\SinOf{\TheApex}} + \SmallRadius
	  \).
	  Then the intersection
	  \[
	    \BallOf[\LargeRadius]{\ThePoint}
	    \,\intersect\,
	    \TheCone[\ThePoint]
	    \,\intersect\,
	    \BallOf[\LargeRadius]{\AltPoint}
	    \,\intersect\,
	    \TheCone[\AltPoint]
	  \]
    contains a lattice point.
  \end{enumerate}
  \end{lemma}
  \begin{proof}
  Within distance $\frac{\SmallRadius+\SqrtOf{\TheDim}/2}{\SinOf{\TheApex}}$
    of $\ThePoint$, we find a point $\ThrPoint$ with
    \(
      \BallOf[\SmallRadius+\SqrtOf{\TheDim}/2]{\ThrPoint}\subset
      \TheCone[\ThePoint]
    \).
    Within the closed ball of radius $\SqrtOf{\TheDim}/2$ around $\ThrPoint$,
    we find the desired lattice point $\AltPoint$. The second assertion can be 
seen as another way of looking at the same 
phenomenon. 
    By (\ref{discr.ensure-lattice-point}), there is a lattice point
    \(
      \ThrPoint\in\TheLattice
    \)
    with
    \(
      \BallOf[\SmallRadius]{\ThrPoint}\subset
    \BallOf[{\frac{\SmallRadius+\SqrtOf{\TheDim}}{\SinOf{\TheApex}}}]{\ThePoint}
      \intersect
      \TheCone[\ThePoint]
    \).
    Now, $\ThrPoint\in\TheCone[\AltPoint]$ since
    $\TheCone[\AltPoint]$ is obtained from
    $\TheCone[\ThePoint]$ via translation by a distance
    less than $\SmallRadius$. By triangle inequality,
    $\ThrPoint\in\BallOf[\LargeRadius]{\AltPoint}$.
  \end{proof}

  A quantitative version of \autoref{cont.connect-two-of-same-type}
  follows immediately.
  \begin{corollary}\label{discr.connect-two-of-same-type}
    Any two lattice points $\ThePoint,\AltPoint\in\TheLattice$
    with $\TheConfOf{\ThePoint}=\TheConfOf{\AltPoint}=\TheDoubleCone$
    and of distance less than $\SmallRadius$
    are connected via a path of two edges of length less than
$\LargeRadius=\frac{\SmallRadius+\SqrtOf{\TheDim}}{\SinOf{\TheApex}}
+\SmallRadius$.
  \end{corollary}

\begin{definition}  
  For $\SmallRadius\leq\LargeRadius$, we call a lattice point
  $\ThePoint\in\TheLattice$
  \notion{$\SmallRadius$-$\LargeRadius$-connected}, if any lattice
  point $\AltPoint\in\BallOf[\SmallRadius]{\ThePoint}$ is connected
  in $\DirGraph$ to $\ThePoint$ via an undirected edge path not
  leaving $\BallOf[\LargeRadius]{\ThePoint}$.
\end{definition}
The following lemma is the discrete version of the
density of well-connected points 
(\autoref{lem:observations} (\ref{cont.well-connected-are-dense})).

  \begin{lemma}\label{discr.r-R-connected-are-dense}
    For any $\SmallRadius\geq 0$, any
    \(
      \LargeRadius > \frac{\SqrtOf{\TheDim} + \SmallRadius}{\SinOf{\TheApex}}
    \),
    and any lattice point $\ThePoint\in\TheLattice$, there is an
    $\SmallRadius$-$\LargeRadius$-connected lattice point
    $\AltPoint\in\BallOf[\LargeRadius]{\ThePoint}$.
  \end{lemma}
  \begin{proof}
    The larger radius $\LargeRadius$ is chosen so that
    \(
      \BallOf[\LargeRadius]{\ThePoint}
      \intersect
      (\BasedCone{\ThePoint})
    \)
    contains a lattice point $\AltPoint$ whose $\SmallRadius$-ball
    $\BallOf[\SmallRadius]{\AltPoint}$ lies within the double
    cone $\BasedCone{\ThePoint}$. Thus any two points
    in $\BallOf[\SmallRadius]{\AltPoint}$ are connected via $\ThePoint$,
    and $\ThePoint$ is within distance $\LargeRadius$ from $\AltPoint$.
  \end{proof}

  Our discrete variant of \autoref{cont.ueber-bande}
  reads as follows:
  \begin{lemma}\label{discr.ueber-bande}
    Consider two lattice points $\ThePoint,\AltPoint\in\TheLattice$ of
    distance less than $\SmallRadius$. Let 
$\TheDoubleCone=\TheConfOf{\AltPoint}$
    be the cone type of $\AltPoint$ and let
    \(
      \LargeRadius>\Enlarge{\SmallRadius}
    \).
    Assume that
    \(
      \BallOf[\SmallRadius]{\ThePoint}\intersect
      \TheDoubleCone[\ThePoint]
    \)
    contains a lattice point $\ThrPoint$ of type $\TheDoubleCone$. Then
    there is an edge path from $\AltPoint$ to $\ThePoint$ not leaving
    $\BallOf[\LargeRadius]{\ThePoint}$.
  \end{lemma}
  \begin{proof}
    There is a directed edge from $\ThrPoint$ to $\ThePoint$.
    Note that the distance of $\AltPoint$ and $\ThrPoint$ is at most
    $2\SmallRadius$. Hence $\LargeRadius$ is chosen so that the triple 
intersection
    \[
      \BallOf[\LargeRadius]{\ThePoint}
      \intersect
      \TheDoubleCone[\ThrPoint]
      \intersect
      \TheDoubleCone[\AltPoint]
      \,\,=\,\,
      \BallOf[\LargeRadius]{\ThePoint}
      \intersect
      \BasedCone{\ThrPoint}
      \intersect
      \BasedCone{\AltPoint}
    \]
    contains a lattice point. Through this point, $\ThrPoint$ and $\AltPoint$
    are connected.
  \end{proof}
  
 The assertion of the following lemma is obvious. 
  \begin{lemma} There is a constant $\JumpMax>0$, depending only on
  $\TheApex$ and the dimension $\TheDim$, such that for any double
  cone $\TheDoubleCone\in\TheConfOf{\TheLattice}$ the following condition
  holds:
  \begin{quote}
    If for a lattice point $\ThePoint\in\TheDoubleCone$, there is
    a lattice point in $\TheDoubleCone$ closer to $0$, then there
    is such a lattice point in 
$\TheDoubleCone\intersect\BallOf[\JumpMax]{\ThePoint}$.
  \end{quote}
  I.e., we can go from $\ThePoint$ within $\TheDoubleCone$ to a lattice
  point of minimum distance to the tip via a chain of jumps each bounded
  in length from above by $\JumpMax$.
  \end{lemma}

 \subsection{The Induction}
  It is our aim, to prove that every two lattice points in a given ball of 
radius $\SmallRadius$ are connected via an edge path that does not leave a 
larger ball of radius $\LargeRadius$. Here, the radius $\LargeRadius$ shall 
depend only on $\SmallRadius$, $\TheApex$ and $\TheDim$. In the following 
lemma, 
we show this for a series of values for $\SmallRadius$ respectively 
$\LargeRadius$. The proof is similar to the proof of the corresponding result 
in the continuous setting, cf. \autoref{cont.connectivity}.
 
\begin{lemma}\label{discr.core-induction}
    There are constants
    \(
      \SmallRadius[1]\leq\ChromaticBound[1]\leq\LargeRadius[1],
      \SmallRadius[2]\leq\ChromaticBound[2]\leq\LargeRadius[2],
      \ldots
    \),
    depending only on $\TheApex$ and $\TheDim$, with
    \(
      \JumpMax
      < \SmallRadius[1]
    \)
    and 
    \( 
      \SmallRadius[i]<\SmallRadius[i+1],
      \ChromaticBound[i]<\ChromaticBound[i+1],
      \LargeRadius[i]<\LargeRadius[i+1]
     \) for every $i \in \N$
    such that any lattice point $\ThePoint\in\TheLattice$ is
    $\SmallRadius[\NumTypes]$-$\LargeRadius[\NumTypes]$-connected provided
    at most $\NumTypes$ cone types are realized at the lattices points
    in $\BallOf[{\ChromaticBound[\NumTypes]}]{\ThePoint}$.
  \end{lemma}
  \begin{proof}
    We induct on $\NumTypes$. The case $\NumTypes=1$ follows directly from
    \autoref{discr.connect-two-of-same-type}: choose
    $\ChromaticBound[1]=\SmallRadius[1]>\JumpMax$
    and
    \(
      \LargeRadius[1] >
      \frac{\SmallRadius[1]+\SqrtOf{\TheDim}}{\SinOf{\TheApex}}
    \).

    For the induction step, assume that constants up to 
$\ChromaticBound[\NumTypes-1]$,
    $\SmallRadius[\NumTypes-1]$, and $\LargeRadius[\NumTypes-1]$ have already 
been found.
    Choose:
    \[
      \AltSmall \,\,>\,\, 
\frac{\ChromaticBound[\NumTypes-1]+\SqrtOf{\TheDim}}{\SinOf{\TheApex}}
      \qquad\text{and}\qquad
      \AltLarge \,\,>\,\, \frac{\AltSmall+\SqrtOf{\TheDim}}{\SinOf{\TheApex}}
    \]
    Note that by \autoref{lem:discrObservations} 
(\ref{discr.ensure-lattice-point}) any
    set
    \(
      \BallOf[\AltSmall]{\WellConnectedPt}\intersect
      \TheDoubleCone[\WellConnectedPt]
    \)
    contains a lattice point $\ThrPoint$ with
    \(
      \BallOf[{\ChromaticBound[\NumTypes-1]}]{\ThrPoint}
      \subset
       \TheDoubleCone[\WellConnectedPt]
    \).
    If $\WellConnectedPt$ is $\AltSmall$-$\AltLarge$-connected, there is an
    edge path from $\WellConnectedPt$ to $\ThrPoint$ not
    leaving $\BallOf[\AltLarge]{\WellConnectedPt}$.

    So we put $\SmallRadius[\NumTypes]=\AltLarge$. By
    \autoref{discr.r-R-connected-are-dense}, there is an
    $\AltSmall$-$\AltLarge$-connected lattice point
    $\WellConnectedPt\in\BallOf[{\SmallRadius[\NumTypes]}]{\ThePoint}$.
    Consider an arbitrary point
    $\AltPoint\in\BallOf[{\SmallRadius[\NumTypes]}]{\ThePoint}$.
    It suffices to choose
    $\LargeRadius[\NumTypes]$ and $\ChromaticBound[\NumTypes]$ so that
    we can ensure the existence of an edge path from $\WellConnectedPt$
    to $\AltPoint$ within $\BallOf[{\LargeRadius[\NumTypes]}]{\ThePoint}$.

    Let $\TheDoubleCone=\TheConfOf{\AltPoint}$ be the cone type of $\AltPoint$.
    The distance of $\AltPoint$ and $\WellConnectedPt$ is less than
    $2\AltLarge$. We are interested in the double half cone
    \(
      \WellConnectedPt
      +
      \ShrinkDoubleCone{\ChromaticBound[\NumTypes-1]}
    \).
    Either tip of the double half cone is within distance
    $\AltSmall<\AltLarge$ of $\WellConnectedPt$ and thus
    within distance $3\AltLarge$ of $\AltPoint$. By
    \autoref{lem:discrObservations} 
(\ref{discr.lattice-point-in-intersection}), 
the
    intersection
    \[
      \BallOf[\WellSmall]{\WellConnectedPt}
      \,\intersect\,
      (\WellConnectedPt+\ShrinkDoubleCone{\ChromaticBound[\NumTypes-1]})
      \,\intersect\,
      \TheDoubleCone[\AltPoint]
    \]
    contains a lattice point for any
    \(
      \WellSmall >
      \frac{
        3\AltLarge+\SqrtOf{\TheDim}
      }{
        \SinOf{\TheApex}
      }
      +
      3\AltLarge
    \).
    
    Choosing
    \(
      \ChromaticBound[\NumTypes]
      >
      \AltLarge+\WellSmall+\ChromaticBound[\NumTypes-1]
    \)
    we can use the induction hypothesis as follows.
    If no lattice point in the region
    \(
      \BallOf[{\WellSmall+\ChromaticBound[\NumTypes-1]}]{\WellConnectedPt}
      \intersect
      \TheDoubleCone[\WellConnectedPt]
      \subset
      \BallOf[{\ChromaticBound[\NumTypes]}]{\ThePoint}
    \)
    is of cone type $\TheDoubleCone$, we see that there are
    at most $\NumTypes-1$ different cone types realized within
    \(
      \BallOf[{\WellSmall+\ChromaticBound[\NumTypes-1]}]{\WellConnectedPt}
      \intersect
      \TheDoubleCone[\WellConnectedPt]
    \).
    Hence, each lattice point in
    \(
      \BallOf[\WellSmall]{\WellConnectedPt}
      \intersect
      (\WellConnectedPt+\ShrinkDoubleCone{\ChromaticBound[\NumTypes-1]})
    \)
    is $\SmallRadius[\NumTypes-1]$-$\LargeRadius[\NumTypes-1]$-connected.
    Since $\SmallRadius[\NumTypes-1]>\JumpMax$, all these well-connected
    balls overlap and are therefore connected to a lattice point
    $\ThrPoint$ near the tip of the double half cone. Recall that
    $\ThrPoint$ is within distance $\AltSmall$ of $\WellConnectedPt$ and
    that $\WellConnectedPt$ is $\AltSmall$-$\AltLarge$-connected. Hence all
    the lattice points in
    \(
      \BallOf[\WellSmall]{\WellConnectedPt}
      \intersect
      (\WellConnectedPt+\ShrinkDoubleCone{\ChromaticBound[\NumTypes-1]})
    \)
    are connected to $\WellConnectedPt$.
    
    On the other hand, one of these lattice points
    lies within the double cone $\AltPoint+\TheDoubleCone
    =\BasedCone{\AltPoint}$ and is hence directly connected to $\AltPoint$.
    Thus, $\AltPoint$ is connected to $\WellConnectedPt$. Each edge path
    used will take us at most $\AltLarge$ or $\LargeRadius[\NumTypes-1]$
    outside of $\BallOf[\WellSmall]{\WellConnectedPt}$.
    Thus, we might choose
    \(
      \LargeRadius[\NumTypes]
      >
      2\AltLarge+\LargeRadius[\NumTypes-1]+\WellSmall
    \).
    We might need to increase this number, to ensure
    \(
      \ChromaticBound[\NumTypes]
      \leq
      \LargeRadius[\NumTypes]
    \),
    but the increase incurred in treating the remaining case is much
    worse.

    It remains to deal with the possibility that there is a lattice point
    of type $\TheDoubleCone$ in the region
    \(
      \BallOf[{\WellSmall+\ChromaticBound[\NumTypes-1]}]{\WellConnectedPt}
      \intersect
      \TheDoubleCone[\WellConnectedPt]
    \).
    Since $\WellConnectedPt$ and $\AltPoint$ are of distance at most
    $2\AltLarge<\WellSmall$, \autoref{discr.ueber-bande}
    applies and we choose
    \(
      \LargeRadius[\NumTypes]
      >
      \Enlarge{(\WellSmall+\ChromaticBound[\NumTypes-1])}
    \).
  \end{proof}
\begin{corollary}\label{discrete--template}
    For every $\SmallRadius>0$ there is $\LargeRadius\geq \SmallRadius$, 
depending only
    on $\SmallRadius$, $\TheApex$ and $\TheDim$, such that for any configuration
    $\TheConf\mapcolon\TheLattice\rightarrow\SetDoubleCones$ with apex
    angles bounded from below by $\TheApex$ any lattice point
    $\ThePoint\in\TheLattice$ is $\SmallRadius$-$\LargeRadius$-connected.
  \end{corollary}
  
  \begin{proof}
    By \autoref{cor:ref-config} we can assume WLOG that 
$\#\TheConf(\TheLattice)=L$, where $L$ is a constant that depends only on 
$\TheApex$ and $\TheDim$. Now the claim follows from 
\autoref{discr.core-induction} and the following observation: 
    If $\ThePoint$ is
    $\SmallRadius$-$\LargeRadius$-connected, it is
    $\SmallerRadius$-$\LargeRadius$-connected for
    any $\SmallerRadius\leq\SmallRadius$.
  \end{proof}

\subsection{Renormalization: Blocks and Towns}
Since the proof of \autoref{theo:main_discrete} involves a renormalization 
argument, it is important to restate \autoref{discrete--template} for 
structures at large scale (see \autoref{renormalization--}).
To this end, we introduce what we call blocks and towns. Recall our notation
  \[
    \CubeOf[\TheLength]{\ThePoint} \,\,=\,\,
    \SetOf[ \AltPoint\in\TheSpace ]{
      \InfNormOf{\AltPoint-\ThePoint}
      \leq \frac{\TheLength}{2}
    }
  \]
  for cubes.
  \begin{lemma}\label{renormalization--apex-shrink}
    For any apex angle $\TheApex$, there is a constant
    $\TheCubeDistance=\TheCubeDistanceOf{\TheApex}$ such that the
    following holds for any points $\ThePoint,\AltPoint$ of distance
    at least $\TheCubeDistance\TheLength$:
    \begin{quote}
      If $\TheCone$ is a cone of apex angle $\frac{\TheApex}{2}$ and
      $\AltPoint\in \TheCone[\ThePoint]$, then
      \[
        \CubeOf[\TheLength]{\AltPoint}
        \,\,\subset\,\,
        \bigcap_{\ThrPoint\in\CubeOf[\TheLength]{\ThePoint}}
          \LargeCone[\ThrPoint]
      \]
      for the cone $\LargeCone$ with apex angle $\TheApex$ and the same
      axis as $\TheCone$.
    \end{quote}
  \end{lemma}
  \begin{proof}
For any $\ThePoint \in \TheLattice$ we have 
\[ \BBB_{(\TheLength\sqrt{\TheDim})/2}(\ThePoint) \,\,\supset\,\,
\CubeOf[\TheLength]{\ThePoint}.\]

Let $\TheCone$ be a cone of apex angle $\frac{\TheApex}{2}$ and symmetry axis 
$v$ 
and let $\LargeCone$ be a cone with apex angle $\TheApex$ and symmetry axis 
$v$. 
Let $\ThePoint \in \TheLattice$. If $\AltPoint \in \TheCone[\ThePoint]$ and 
$\DistOf{\ThePoint}{\AltPoint}\geq\frac{\TheLength 
\sqrt{\TheDim}}{2\SinOf{\TheApex}}$, then
\[ \BBB_{(\TheLength\sqrt{\TheDim})/2}(\AltPoint) \,\,\subset\,\,
\LargeCone[\ThePoint].\]

According to \autoref{lem:cone in intersection}
\[
  \bigcap_{\ThrPoint\in 
\CubeOf[\TheLength]{\ThePoint}}\LargeCone[\ThrPoint]\,\,\supset\,\,
\LargeCone_{\frac{\TheLength}{2}\sqrt{\TheDim}}[\ThePoint] \,\,\,\,\,\text{and}\,\,\,\,\,
\LargeCone_{\TheLength 
\sqrt{\TheDim}}[\ThrPoint]\,\,\subset\,\,\,
\LargeCone_{\frac{\TheLength}{2}\sqrt{\TheDim}}[\ThePoint] \quad\text{for every } 
\ThrPoint 
\in \CubeOf[\TheLength]{\ThePoint}.\]
Therefore we choose $\TheCubeDistance \geq \frac{ 
\sqrt{\TheDim}}{2\SinOf{\TheApex}} + 
\frac{\sqrt{\TheDim}}{\SinOf{\TheApex}} = 
\frac{3\sqrt{\TheDim}}{2\SinOf{\TheApex}}$. 

\end{proof}
  
  \begin{definition}
  A \notion{block} 
  \[
    \BlockOf[\TheLength]{\ThePoint}
    \,\,=\,\,
    \TheLattice
    \intersect
    \CubeOf[\TheLength]{\ThePoint}
  \]
  is a collection of lattice points
  inside a cube. The \notion{town at scale} $(\TheScale,\TheLength)$
  is the collection
  \[
    \TownOf{\TheScale,\TheLength}
    \,\,=\,\,
    \SetOf[{ \BlockOf[\TheLength]{\ThePoint} }]{ 
\ThePoint\in\TheScale\TheLattice }
  \]
  If the constant $\TheCubeDistance$ from \autoref{renormalization--apex-shrink}
  is less than $\frac{\TheScale}{\TheLength}$, we call the town
  \notion{sparsely populated} (or $\TheApex$-sparsely populated when
  we want to recall that $\TheCubeDistance$ depends on $\TheApex$).
  \end{definition}
  
  In order to employ geometric language, we implicitly may identify the block
  $\BlockOf[\TheLength]{\ThePoint}$ with its center $\ThePoint$. This way,
  we think of the distance between two blocks as the distance of their
  centers. If $\TheScale$ is large compared to $\TheLength$, the distance
  between the centers is a good approximation to any distance between points
  from the two blocks.
\begin{definition}
  Let $\TheConf\mapcolon\TheLattice\rightarrow\SetDoubleCones$ be
  a $\TheApex$-bounded configuration. We call a double cone $\TheDoubleCone 
\in \SetDoubleCones$
  \notion{favored by majority} in $\TheBlock$ for a block
  $\TheBlock\subset\TheLattice$ if the pre-image
  \(
    \TheConf{}_{\TheBlock}^{-1}(\TheDoubleCone)=
      \SetOf[
        \ThePoint\in\TheBlock
      ]{
        \TheConfOf{\ThePoint}=\TheDoubleCone
      }
  \)
  has maximal size, i.e.,
  \[
    \CardOf{\TheConf{}_{\TheBlock}^{-1}(\TheDoubleCone)}
    \,\,\geq\,\,
    \CardOf{\TheConf{}_{\TheBlock}^{-1}(\AltDoubleCone)}
    \qquad\text{for every }
    \AltDoubleCone\in\SetDoubleCones
  .\]
  \end{definition}
  
  \begin{remark}
   Given a block $\TheBlock$, the choice of a cone $\TheDoubleCone \in 
\SetDoubleCones$ that is favored by majority in $\TheBlock$, in general, is not 
unique. 
  \end{remark}

  \begin{definition}
  Given a town $\TheTown=\TownOf{\TheScale,\TheLength}$, we now define a
  directed graph as follows. The vertices are given by the blocks in 
$\TheTown$. There is an edge from a block $\TheBlock$
  to a block $\AltBlock$ if there is a cone $\TheDoubleCone \in 
\SetDoubleCones$ favored
  by majority in $\TheBlock$ with:
  \[
    \AltPoint \in \TheDoubleCone{[\ThePoint]}
    \qquad\text{for all\ }\ThePoint\in\TheBlock,\,\,\AltPoint\in\AltBlock
  \]
  We call the corresponding undirected graph the \notion{favored graph}.
\end{definition}
  
    We derive a connectivity result for the favored graph of a
  sparsely populated town from \autoref{discrete--template}.
  
  \begin{proposition}\label{renormalization--}
    For any radius $\SmallRadius>0$, there exists 
$\LargeRadius\geq \SmallRadius$ depending
    only on $\TheApex$ and $\TheDim$, such that in a $\TheApex$-sparsely
    populated town $\TheTown$ of scale $(\TheScale,\TheLength)$ any two
    blocks $\TheBlock$ and $\AltBlock$ within distance $\TheScale\SmallRadius$
    of some point $\ThrPoint\in\TheScale\TheLattice$ are connected
    by an undirected edge path in the favored graph that does not pass
    through blocks farther away from $\ThrPoint$ than $\TheScale\LargeRadius$.
  \end{proposition}

\begin{proof}
  Let $r>0$ and $\TheTown=\TheTown(\TheScale,\TheLength)$ be a sparsely 
populated town. Let $\TheBlock,\AltBlock \in 
\TheTown=\TheTown(\TheScale,\TheLength)$ be two blocks within distance $hr$ of 
some point $\ThrPoint \in \TheScale\TheLattice$.
    Denote by $\TownConfOf{\TheBlock}\in\SetDoubleCones$ one of the cones that 
are favored by majority in $\TheBlock$. 
  Let us show the existence of a path in the favored graph that connects 
$\TheBlock$ and $\AltBlock$ and does not leave the ball $B_{\TheScale 
\LargeRadius}(\ThrPoint)$.
  In order to invoke \autoref{discrete--template},  
 note that
  \begin{align*}
    \TheLattice &\longrightarrow\TownOf{\TheScale,\TheLength} \\
    \ThePoint & \mapsto \BlockOf[\TheLength]{\TheScale\ThePoint}
  \end{align*}
  provides an identification of the town $\TheTown$ with the integer lattice
  $\TheLattice$. Denoting by $\TownConfOf[\frac12]{\TheBlock}$ the double
  cone with apex $\frac{\TheApex}{2}$ and the same axis as 
$\TownConfOf{\TheBlock}$,
  let us consider the following configuration:
  \begin{align*}
    \TheLattice &\longrightarrow\SetDoubleCones[\frac{\TheApex}{2}] \\
    \ThePoint & \mapsto 
\TownConfOf[\frac12]{\BlockOf[\TheLength]{\TheScale\ThePoint}}.
  \end{align*}
  If there is an edge from $\ThePoint$ to $\AltPoint$ in this configuration,
  then by \autoref{renormalization--apex-shrink}, there is an edge from the 
block $\BlockOf[\TheLength]{\TheScale\ThePoint}$
  to the block $\BlockOf[\TheLength]{\TheScale\AltPoint}$ in the favored 
graph. 
  Choose $\ThePoint,\AltPoint \in \TheLattice$ so that $P=Q_\ell(h\ThePoint)$ 
and $Q=Q_\ell(h\AltPoint)$.  \\
Now the claim follows from \autoref{discrete--template}.
  \end{proof}

\subsection{Connecting Points at Scale}
  From \autoref{discrete--template} it is clear that, for any
  configuration
  \(
    \TheConf\mapcolon\TheLattice\rightarrow\SetDoubleCones
  \)
  with apex angles bounded away from $0$,
  the associated directed graph $\DirGraph=\DirGraphOf{\TheConf}$ is
  connected when considered as an undirected graph. Thus, there is a set of 
paths, which is large enough to connect any given pair $\ThePoint, \AltPoint$.
  The aim of this section is to prove quantitative estimates on the length of 
paths
  and the number of edges. The following contains our main result in this 
direction. As shown below, it implies \autoref{theo:main_discrete} quite 
directly:
  \begin{theorem}\label{theo:path-props}
    Let
    \(
      \TheConf\mapcolon\TheLattice\rightarrow\SetDoubleCones
    \)
    be a configuration with apex angles bounded from below by $\TheApex>0$.
    Let $\minDist>0$. There exist positive numbers $\BdNumJumps$ and 
$\BdEdgeUsage$ and a constant $\StepComparability\geq \minDist$, all 
independent 
of $\TheConf$, and a collection
    \(
      \FamOf[{
        \ThePath[\ThePoint\AltPoint]
      }]{
        \ThePoint,\AltPoint\in\TheLattice
      }
    \)
    of unoriented edge paths in $\DirGraph$ such that the following holds:
    \begin{enumerate}
      \item
        The path $\ThePath[\ThePoint\AltPoint]$ starts at $\ThePoint$
        and ends at $\AltPoint$.
      \item
        Any path $\ThePath[\ThePoint\AltPoint]$ has at most
        $\BdNumJumps$ edges.
      \item
        Any edge of $\DirGraph$ is used in at most $\BdEdgeUsage$
        paths $\ThePath[\ThePoint\AltPoint]$.
      \item 
      Any edge in $\ThePath[\ThePoint\AltPoint]$ has length comparable to 
$\DistOf{\ThePoint}{\AltPoint}$ with constant $\StepComparability$.
    
    \end{enumerate}
  \end{theorem}

Let us provide the setup of the proof of \autoref{theo:path-props}. We pick an
even integer $\TheScaleStep$ larger than the constant $\max(\TheCubeDistance, 
\minDist)$ and with the property that $\TheScaleStep/L \in \N$, where 
$\TheCubeDistance$ is as in
  \autoref{renormalization--apex-shrink} and $L$ is as in 
\autoref{lem:ref-cones}. 
Hence, the towns
  $\TheTown[\LogScale] =
  \TownOf{\TheScaleStep^{\LogScale},\TheScaleStep^{\LogScale-1}}$ are all
  $\TheApex$-sparsely populated so that \autoref{renormalization--}
  applies. The distance $\DistOf{\ThePoint}{\AltPoint}$ lies in exactly one
  of the intervals
  \(
    [\TheScaleStep^{0},\TheScaleStep^{1})
  \),
  \(
    [\TheScaleStep^{1},\TheScaleStep^{2})
  \),
  \(
    [\TheScaleStep^{2},\TheScaleStep^{3})
  \), etc., say:
  \(
    \DistOf{\ThePoint}{\AltPoint}
    \in
    [\TheScaleStep^{\LogScale-1},\TheScaleStep^{\LogScale})
  \).
  In this case, we will consider $\TheTown[\LogScale]$ to be the appropriate
  town for connecting $\ThePoint$ and $\AltPoint$. We call $\LogScale$ the
  \notion{logarithmic scale} of the town $\TheTown[\LogScale]$. Let $L$ denote 
the same constant as in \autoref{lem:ref-cones}.
  Assume that $\#\Gamma(\Z^d)\leq L$. Since $\TheScaleStep$ is an even
  integer, each block 
$\TheBlock=\BlockOf[\TheScaleStep^{\LogScale-1}]{\ThePoint}$
  contains at least
  \(
    \frac{\TheScaleStep^{\TheDim(\LogScale-1)}}{L}
  \)
  lattice points $\ThrPoint\in\TheBlock$ where the associated cone
  $\TheConfOf{\ThrPoint}$ is favored by majority in $\TheBlock$. 
  
  An important step in the construction of $\ThePath[\ThePoint\AltPoint]$ is to 
  connect $\ThePoint$ and $\AltPoint$ to blocks of
  $\TheTown[\LogScale]$. The following lemma deals with this problem.

  \begin{lemma}\label{connect--first-jump}
    There is a constant $\FirstJump\geq 1$ such that for
    any point $\ThePoint \in \TheSpace$ and any $n \in \mathbbm{N}$ there is a 
    block $\TheBlock\in\TheTown[\LogScale]$
    entirely contained in
    \(
      \BallOf[\TheScaleStep^{\LogScale}\FirstJump]{\ThePoint}
      \intersect
      \ThePoint+\TheConfOf{\ThePoint}
    \).
  \end{lemma}
  \begin{proof}
    There is a radius $\TheRadius$ such that for any cone $\TheCone$ of
    apex at least $\frac{\TheApex}{2}$ and each point $\ThrPoint\in\TheSpace$,
    the intersection $\BallOf[\TheRadius]{\ThrPoint}\intersect 
\TheCone[\ThrPoint]$
    contains a lattice point $\AltPoint$. Now the claim follows by rescaling 
from
    \autoref{renormalization--apex-shrink} applied to 
    $\TheScaleStep^{\LogScale} \ThrPoint$ and $\TheScaleStep^{\LogScale} 
\AltPoint$. As we want to encircle the whole
    block and not just its center, we choose $\FirstJump > 
\TheRadius+\sqrt{\TheDim}$.
\end{proof}
  
  Now we are in the position to prove the main result of this section. 
  
\begin{proof}[Proof of \autoref{theo:path-props}]
  Let $\FirstJump$ be the radius from~\autoref{connect--first-jump}, put
  $r=2\sqrt{d}+R_1$, and let $\LargeRadius$ be the radius resulting with this
  value from~\autoref{renormalization--}.
  The proof consists of several steps.

  \emph{Step 1: Construction of paths in the favored graph for a fixed scale.}  
We fix some logarithmic scale $\LogScale$. For every $\ThrPoint \in 
\TheScaleLog 
\TheLattice$, we construct a path $\mathcal{P}^{\LogScale}_{\ThrPoint}$ in the 
favored graph that traverses every block of $\TheTown[\LogScale]$ that is a 
subset of $B_{\TheScaleLog \SmallRadius}(\ThrPoint)$. By taking the union 
$\bigcup_{\ThrPoint \in \TheScaleLog \TheLattice} 
\mathcal{P}^{\LogScale}_{\ThrPoint}$ we construct paths in the favored graph 
for a fixed scale. Let $\ThrPoint \in \TheScaleLog \TheLattice$.
  \autoref{renormalization--} allows us to connect every block $Q \in 
\TheTown[\LogScale], Q\subset B_{\TheScaleLog\SmallRadius}(\ThrPoint)$ with 
every other block $P \in \TheTown[\LogScale], P \subset 
B_{\TheScaleLog\SmallRadius}(\ThrPoint)$ so that the corresponding path 
traverses not more than $\#(B_{\LargeRadius} \cap \Z^d) \asymp R^d$ blocks of 
$\TheTown[\LogScale]$, which can be chosen to lie in $\BBB_{\TheScaleLog 
\LargeRadius}$.
  If we apply \autoref{renormalization--} successively to all blocks of 
$\TheTown[\LogScale]$, which are subsets of $B_{\TheScaleLog \SmallRadius}$, 
then we obtain a path  

\begin{align}\label{bigpath} 
\mathcal{P}^{\LogScale}_{\ThrPoint} \,\,=\,\, Q^1 - Q^2 
- ... -Q^t 
\end{align}

in the favored graph with 
\[ 
r^d \,\,\asymp\,\,
\#(B_{\TheScaleLog \SmallRadius} \cap \TheLattice) \,\,\leq\,\,
t\,\,\leq\,\,  
\#(B_{\TheScaleLog \SmallRadius} \cap \TheLattice)(B_{\TheScaleLog 
\LargeRadius} 
\cap \TheLattice) \,\,\asymp\,\, r^d R^d 
\]
 of blocks of $\TheTown[\LogScale]$ in the favored graph such that the 
following holds:

\begin{enumerate}
  \item For each $i \in \{1,...,t\}$ we have $Q^i \subset B_{\TheScaleLog 
R}(z)$,
  \item The blocks $Q^1$ and $Q^t$ are subsets of $B_{\TheScaleLog r}(z)$,
  \item If $Q$ is any block of $\TheTown[\LogScale]$ 
    with $Q\subset B_{\TheScaleLog r}(z)$, then $Q=Q^i$ for some $i\in 
\{1,...,t\}$. 
\end{enumerate} 
   \begin{figure}[ht]
  \centering
    \scalebox{.44}{\input{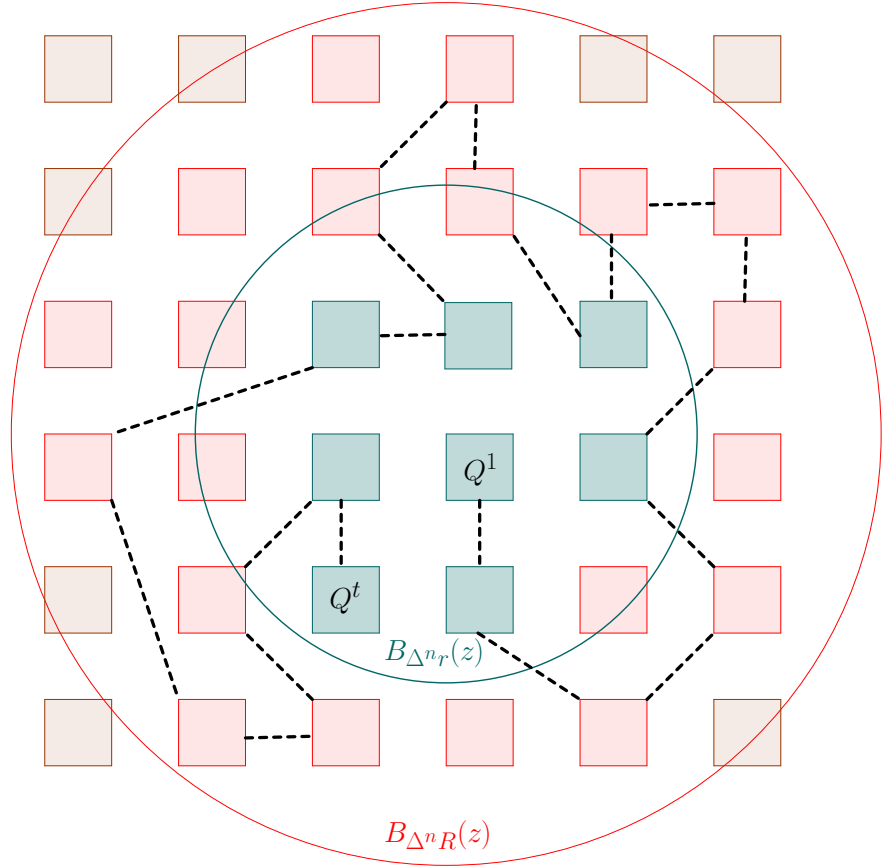}}
  \caption{The path $Q_1-....-Q_t$}
  \end{figure}
  Finally, set 
  $\mathcal{P}^{\LogScale}=\bigcup_{\ThrPoint \in \TheScaleLog \TheLattice} 
\mathcal{P}^{\LogScale}_{\ThrPoint}$.
  
\emph{Step 2: Construction of paths in the graph $G$ for a fixed scale.}
 For a logarithmic scale $\LogScale$ and $\ThePoint, \AltPoint \in 
[\TheScaleLogmin,\TheScaleLog)$ we construct a path in the graph $G$ connecting 
$\ThePoint$ and $\AltPoint$. 
     
  Fix a logarithmic scale $\LogScale$. Let $z \in \Delta^n \TheLattice$.
  Choose for every block in \eqref{bigpath} a favored cone and call the 
corresponding set of points in the block where this cone is associated 
\emph{majority set}. Each majority set contains at least 
  \[ a\,\,=\,\, \frac{\Delta^{d(n-1)}}{L} \,\,\in\,\N \]
  points. WLOG we assume that every majority set contains exactly $a$ different 
elements. Then we identify a block in \eqref{bigpath} with its majority set, 
i.e., if $Q^k$ is the $k$-th block in \eqref{bigpath}, then
  \[ Q^k\,\,=\,\,(q^k_i)_{1\leq i \leq a}.\]

  Starting from \eqref{bigpath} we now fix certain paths in the graph $G$, 
which then give rise to the collection $(p_{xy})$.
  Let $i \in \{1,...,a\}$. WLOG we assume that $t$ is an even number (for odd 
$t$ just erase the last edge in the following scheme). Let $M$ be the set of 
the 
following paths in $G$:
  \begin{align}\label{constrpathes} \begin{matrix}
  q^1_i &-	&q^{2}_i 	&- 	&q^{3}_i	&-	&q^{4}_i	
&-	&\dots	&-	&q^{t}_i \\
  q^1_i &-	&q^{2}_{i+1}	&-	&q^{3}_{i}	&-	&q^{4}_{i+1}	
&-	&\dots 	&- 	&q^{t}_{i+1}  \\
  q^1_i &-	&q^{2}_{i+2}	&-	&q^{3}_{i}	&-	&q^{4}_{i+2} 	
&-	&\dots 	&- 	&q^{t}_{i+2} \\
  q^1_i &- 	&q^{2}_{i+3} 	&- 	&q^{3}_{i} 	&-	&q^{4}_{i+3} 	
&-  	&\dots	&- 	&q^{t}_{i+3} \\
  \vdots&        &\vdots		&	&\vdots		&\	&\vdots		
&\	& \vdots&\	&\vdots \\	
  q^1_i &- 	&q^{2}_{i+a-1} 	&- 	&q^{3}_i 	&- 	&q^{4}_{i+a-1} 	
&-  	&\dots 	&- 	&q^{t}_{i+a-1}.
  \end{matrix} \end{align}
  
 Here, the lower index is to be read modulo $a$, i.e., $k+a=k$ for every 
$k$. Since we do this for every $i \in \{1,...,a\}$, the set $M$ consists of 
$a^2$ paths. Now we associate to every pair $(\ThePoint, \AltPoint)\in$ 
\[ 
  A\,\,=\,\,
  \{\,
    (x,y) \in B_{\TheScaleLog 2\sqrt{\TheDim} }(z) \times B_{\TheScaleLog 2\sqrt{\TheDim} }(z)
    \,\,|\,\,
    |x-y|\in [\TheScaleLogmin, \TheScaleLog)
  \,\}
\] 
one path of $M$. Since the numbers $a$ and $\#A$ are comparable, i.e., its 
ratio is bounded by a number independent of $\LogScale$, this can be realized 
by 
a function 
\[
  \phi_{\ThrPoint}: A \to M
\]
with
\[
  \# \phi_{\ThrPoint}^{-1}(p)
  \,\,\leq\,\, K 
  \quad\text{for every } p \in M
\] 
where $K\geq 1$ is independent of $\LogScale$ and $p$. 
In order to use the path $\phi_{\ThrPoint}(\ThePoint,\AltPoint)$ to connect 
$\ThePoint$ and $\AltPoint$ it remains to make sure that $\ThePoint$ and 
$\AltPoint$ are both connected in $\DirGraph$ to one element in 
$\phi_{\ThrPoint}(\ThePoint,\AltPoint)$ respectively. This follows from 
\autoref{connect--first-jump} which guarantees that    
   every $\ThePoint \in B_{\TheScaleLog 2\sqrt{\TheDim}}(\ThrPoint)$ is 
connected to every point in some block $Q^k$ of 
$\mathcal{P}^{\LogScale}_{\ThrPoint}$. 
   In this way the path $\phi_{\ThrPoint}(\ThePoint, \AltPoint)$ induces a path 
in $\DirGraph$ that starts in $\ThePoint$ and ends in $\AltPoint$ (cf. 
\autoref{fig:construction-paths}). 
  
\begin{figure}[ht]
\centering
   \ProvidesFile{fig--construction-paths.tex}
\begingroup
\definecolor{ffqqqq}{rgb}{1,0,0}
\definecolor{qqttcc}{rgb}{0,0.2,0.8}
\definecolor{qqqqff}{rgb}{0,0,1}
\definecolor{zzttqq}{rgb}{0.6,0.2,0}
\begin{tikzpicture}[line cap=round,line join=round,>=triangle 45,x=1.0cm,y=1.0cm]
\clip(1.83,-4.13) rectangle (13.71,4.11);
\fill[color=zzttqq,fill=zzttqq,fill opacity=0.1] (6,3) -- (6,4) -- (7,4) -- (7,3) -- cycle;
\fill[color=zzttqq,fill=zzttqq,fill opacity=0.1] (8,2) -- (8,3) -- (9,3) -- (9,2) -- cycle;
\fill[color=zzttqq,fill=zzttqq,fill opacity=0.1] (10,1) -- (10,2) -- (11,2) -- (11,1) -- cycle;
\fill[color=zzttqq,fill=zzttqq,fill opacity=0.1] (10,-2) -- (10,-1) -- (11,-1) -- (11,-2) -- cycle;
\fill[color=zzttqq,fill=zzttqq,fill opacity=0.1] (8,-3) -- (8,-2) -- (9,-2) -- (9,-3) -- cycle;
\fill[color=zzttqq,fill=zzttqq,fill opacity=0.1] (6,-4) -- (6,-3) -- (7,-3) -- (7,-4) -- cycle;
\draw [color=zzttqq] (6,3)-- (6,4);
\draw [color=zzttqq] (6,4)-- (7,4);
\draw [color=zzttqq] (7,4)-- (7,3);
\draw [color=zzttqq] (7,3)-- (6,3);
\draw [color=zzttqq] (8,2)-- (8,3);
\draw [color=zzttqq] (8,3)-- (9,3);
\draw [color=zzttqq] (9,3)-- (9,2);
\draw [color=zzttqq] (9,2)-- (8,2);
\draw [color=zzttqq] (10,1)-- (10,2);
\draw [color=zzttqq] (10,2)-- (11,2);
\draw [color=zzttqq] (11,2)-- (11,1);
\draw [color=zzttqq] (11,1)-- (10,1);
\draw [color=zzttqq] (10,-2)-- (10,-1);
\draw [color=zzttqq] (10,-1)-- (11,-1);
\draw [color=zzttqq] (11,-1)-- (11,-2);
\draw [color=zzttqq] (11,-2)-- (10,-2);
\draw [color=zzttqq] (8,-3)-- (8,-2);
\draw [color=zzttqq] (8,-2)-- (9,-2);
\draw [color=zzttqq] (9,-2)-- (9,-3);
\draw [color=zzttqq] (9,-3)-- (8,-3);
\draw [color=zzttqq] (6,-4)-- (6,-3);
\draw [color=zzttqq] (6,-3)-- (7,-3);
\draw [color=zzttqq] (7,-3)-- (7,-4);
\draw [color=zzttqq] (7,-4)-- (6,-4);
\draw [line width=1.5pt,dash pattern=on 2pt off 2pt,color=qqttcc] (6.32,3.55)-- (8.48,2.15);
\draw [line width=1.5pt,dash pattern=on 2pt off 2pt,color=qqttcc] (8.48,2.15)-- (10.8,1.84);
\draw [line width=1.5pt,dash pattern=on 2pt off 2pt,color=qqttcc] (10.7,-1.76)-- (8.35,-2.66);
\draw [line width=1.5pt,dash pattern=on 2pt off 2pt,color=qqttcc] (8.35,-2.66)-- (6.48,-3.54);
\draw [color=qqttcc](10.41,0.57) node[anchor=north west] {\huge{$\vdots$}};
\draw [line width=1.5pt,dash pattern=on 2pt off 2pt,color=qqttcc] (10.8,1.84)-- (10.69,0.46);
\draw [line width=1.5pt,dash pattern=on 2pt off 2pt,color=qqttcc] (10.67,-0.6)-- (10.7,-1.76);
\draw(4,0) circle (2cm);
\draw [line width=2pt,dash pattern=on 4pt off 4pt,color=ffqqqq] (3.19,1.12)-- (8.48,2.15);
\draw [line width=2pt,dash pattern=on 4pt off 4pt,color=ffqqqq] (4.76,-1.18)-- (8.35,-2.66);
\draw [color=qqttcc](10.87,0.56) node[anchor=north west] {\large{$\phi_{\ThrPoint}(\ThePoint,\AltPoint)$}};
\draw [color=black](2.3,-0.5) node[anchor=north west] {\large{$ B_{\TheScaleLog 2\sqrt{\TheDim}}(\ThrPoint) $}};
\begin{scriptsize}
\fill [color=qqqqff] (6.32,3.55) circle (2pt);
\draw[color=qqqqff] (6.77,3.69) node {\large{$q_{1}^{1}$}};
\fill [color=qqqqff] (8.48,2.15) circle (2pt);
\draw[color=qqqqff] (8.64,2.5) node {\large{$q_{2}^{2}$}};
\fill [color=qqqqff] (10.8,1.84) circle (2pt);
\draw[color=qqqqff] (10.46,1.49) node {\large{$q^{3}_{1}$}};
\fill [color=qqqqff] (10.7,-1.76) circle (2pt);
\draw[color=qqqqff] (10.9,-2.29) node {\large{$q^{t-2}_{2}$}};
\fill [color=qqqqff] (8.35,-2.66) circle (2pt);
\draw[color=qqqqff] (8.45,-2.3) node {\large{$q^{t-1}_{1}$}};
\fill [color=qqqqff] (6.48,-3.54) circle (2pt);
\draw[color=qqqqff] (6.68,-3.73) node {\large{$q^{t}_{2}$}};
\fill [color=ffqqqq] (3.19,1.12) circle (2.5pt);
\draw[color=ffqqqq] (3.32,1.45) node {\large{$\ThePoint$}};
\fill [color=ffqqqq] (4.76,-1.18) circle (2.5pt);
\draw[color=ffqqqq] (4.92,-0.85) node {\large{$\AltPoint$}};
\end{scriptsize}
\end{tikzpicture}
\endgroup
 
\caption{Construction of a path using 
$\phi_z(x,y)$}\label{fig:construction-paths}
\end{figure}
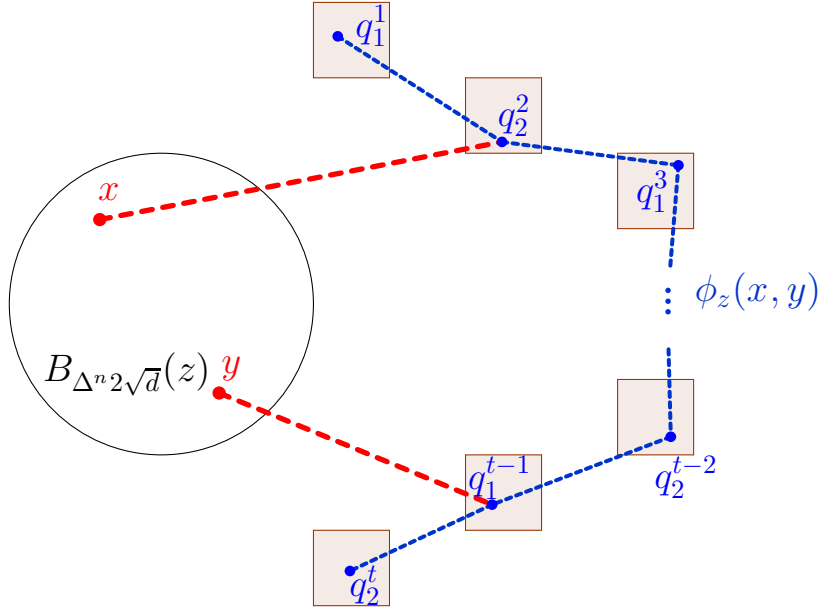
   Using this construction scheme, we have constructed a path for each pair 
$(x,y) \in B_{2\sqrt{\TheDim }\TheScaleLog }(z) \times 
B_{2\sqrt{\TheDim}\TheScaleLog }(z)$ with $|x-y|\in 
[\TheScaleLogmin,\TheScaleLog)$. Let $M_{\ThrPoint}^{\LogScale}$ be the set of 
all these paths. 
   This principle of construction of the paths can be carried out for every $z 
\in \TheScaleLog \Z^d$. 
  
\emph{Step 3: Construction of $p_{\ThePoint \AltPoint}$.}
   Note that the whole construction process of step $2$ has been performed for 
an arbitrary $\LogScale\in \N$.
   We define $p_{\ThePoint \AltPoint}$ for $\ThePoint,\AltPoint \in 
\TheLattice$ 
as follows. Choose $\LogScale \in \N$ such that $\DistOf{\ThePoint}{\AltPoint} 
\in [\TheScaleLogmin, \TheScaleLog)$. Next, choose any $\ThrPoint \in 
\TheLattice$ such that $\phi_{\ThrPoint}(\ThePoint,\AltPoint)$ represents a 
path 
connecting $\ThePoint$ and $\AltPoint$. In this way, $p_{\ThePoint \AltPoint} 
\in  \bigcup_{\ThrPoint\in \TheScaleLog\Z^d} M_{\ThrPoint}^{\LogScale}$. 

\emph{Step 4: Bounds of the length of each edge path.} 
   The second claim of \autoref{theo:path-props} follows immediately from 
   $t \leq \#(B_r\cap \Z^d) \cdot \#(B_R\cap \Z^d)$.
    
\emph{Step 5: Bounds of the length of each edge.}
     By construction, all edges used in $\ThePath[\ThePoint\AltPoint]$ for some 
$\ThePoint, \AltPoint \in \TheLattice$ with $\DistOf{\ThePoint}{\AltPoint} \in 
[\TheScaleStep^{\LogScale-1}, \TheScaleStep^{\LogScale})$ have lengths bounded 
from
    below by $\TheScaleStep^{\LogScale-1}$ and from above by 
$2\TheScaleStep^{\LogScale}\LargeRadius$. 
    Ergo the fourth claim follows with $\StepComparability=2R\TheScaleStep$.
    
\emph{Step 6: Bounds of the multiplicity of edges.}
     According to step 5, it is enough to proof the third claim of 
\autoref{theo:path-props} for one fixed logarithmic scale. Therefore we fix 
$n$. 
Assume $e$ is an edge of length in $ [\TheScaleStep^{\LogScale-1},   
2\TheScaleStep^{\LogScale}\LargeRadius)$.
     Then there exists a point $\ThrPoint \in \TheScaleLog \TheLattice$ so that 
$e \in \BBB_{\TheScaleLog\LargeRadius}(\ThrPoint)$. Since the number of lattice 
points in $\BBB_{2\LargeRadius}$ bounds from above
     the number of block centers 
$\ThrPoint\in\TheScaleStep^{\LogScale}\TheLattice$ for which
     $\BallOf[\TheScaleStep^{\LogScale}\LargeRadius]{\ThrPoint}$ contains $e$, 
it is enough to bound the number of times $e$ is used by paths belonging to a 
fixed $z$. But now by construction (cf. step 1) for every edge in 
$\BBB_{\TheScaleLog\LargeRadius}(z)$ the usage of paths that start in some 
point 
$\ThePoint$ and end in some other point $\AltPoint$ with $\ThePoint,\AltPoint  
\in \BBB_{\TheScaleLog 2\sqrt{\TheDim}}(\ThrPoint)$ so that 
$|\ThePoint-\AltPoint| \in [\TheScaleLogmin,\TheScaleLog)$, is bounded by $K$ 
and this number is independent of the scale.
\end{proof}

\section{Proof of 
\texorpdfstring{\autoref{theo:main_discrete}}{Theorem 
\ref{theo:main_discrete}}}\label{sec:proof_main-result}

We are now in the position to prove \autoref{theo:main_discrete}. The proof is 
just an easy consequence of \autoref{theo:path-props}.

\begin{proof}
 Let $R>0$ and $x_0\in \R^d$. For $x,y \in B_R(x_0)\cap \Z^d$ denote by 
$(x=z_1,z_2,...,z_{N-1},z_N=y)$ 
the path $p_{xy}$ that satisfies properties (1)-(4) of 
\autoref{theo:path-props}. For simplicity we assume here that every path in 
$(p_{xy})$ is of length $N$. Then with use of the properties (1)-(4) of 
\autoref{theo:path-props} and of \eqref{assum:main_discrete} we find:
\begin{align*}
  &\sum\limits_{\underset{|x-y|>R_0}{x,y\in B_R(x_0) \cap \Z^d}} 
   (f(x)-f(y))^2 |x-y|^{-d-\alpha} \\
\leq\,\,&
  2 \lambda^{d+\alpha} \sum\limits_{\underset{|x-y|>R_0}{x,y\in B_R(x_0) \cap \Z^d}}
  \sum_{i=1}^{N-1} (f(z_{i+1})-f(z_i))^2 |z_{i+1}-z_i|^{-d-\alpha}\\
\leq\,\,&
  2 \lambda^{d+\alpha} \sum\limits_{\underset{|x-y|>R_0}{x,y\in B_R(x_0) \cap \Z^d}}
  (N-1) \max\limits_{i \in \{1,...,N-1\}}
  \left[(f(z_{i+1})-f(z_i))^2 |z_{i+1}-z_i|^{-d-\alpha} \right]\\
\leq\,\,&
  2 \Lambda \lambda^{d+\alpha}\sum\limits_{\underset{|x-y|>R_0}{x,y\in B_R(x_0) \cap \Z^d}}
  (N-1) \max\limits_{i \in \{1,...,N-1\}}
  \left[(f(z_{i+1})-f(z_i))^2 \omega(z_{i+1},z_i)\right]\\
=\,\,&
  2 \Lambda \lambda^{d+\alpha}(N-1)M
  \sum\limits_{\underset{|x-y|>R_0}{x,y\in B_{(N-1)\lambda R}(x_0)\cap \Z^d}} 
  (f(x)-f(y))^2 \omega(x,y).
\end{align*}

Therefore $c = \left(2 \Lambda 
\lambda^{d+\alpha}(N-1)M\right)^{-1}$ and $\kappa = 
(N-1)\lambda$.
\end{proof}

 \appendix
 
 \section{Auxiliary results}\label{sec:aux}
The following lemma is a version of \cite[Lemma 6.9]{DyKa15} that matches our 
integral kernels. Note that \cite[Lemma 6.9]{DyKa15} is concerned 
with translation invariant expressions. The proof also applies to our 
case. 
\begin{lemma}\label{lem:new6.9}
Let $\alpha \in (0,2)$ and $\kappa \geq 1$. For $B=B_R(x_0),R>0,x_0\in \R^d$ we 
set $B^\ast = B_{\kappa R}(x_0)$. Let $k:\R^d \times \R^d\to \R$ be a symmetric 
kernel that satisfies \eqref{assum:main}. Suppose that for some $c>0$
 \[ c\int_{B \times B}(f(x)-f(y))^2|x-y|^{-d-\alpha} \dxy \,\,\leq\,\,
\int_{B^\ast \times B^\ast}(f(x)-f(y))^2k(x,y) \dxy \]
for every ball $B\subset \R^d$ and every $f \in H_k(B^\ast)$. Then for every 
bounded Lipschitz domain $\Omega\subset \R^d$ 
there exists a constant $\widetilde{c}=\widetilde{c}(d,\kappa,\alpha, 
\Omega)>0$ 
such 
that for 
every $f\in H_k(\Omega)$
\[ \widetilde{c}c\int_{\Omega \times \Omega}(f(x)-f(y))^2|x-y|^{-d-\alpha} \dxy 
\,\,\leq\,\, 
\int_{\Omega \times \Omega}(f(x)-f(y))^2k(x,y) \dxy. \]
The constant $\widetilde{c}$ depends on the domain
$\Omega$ only up to scaling. In particular, if $\Omega$ is a ball, the constant can
be chosen independently of $\Omega$.
For $0 < \alpha_0 \leq \alpha <2$, the constant $\widetilde{c}$ depends on 
$\alpha_0$ but not on $\alpha$. 
\end{lemma}

\begin{proof}
 Let $\Omega$ be a bounded Lipschitz domain. The Whitney 
decomposition technique provides a family $\mathcal{B}$ of balls with the 
following properties.
\begin{enumerate}
 \item[(i)] There exists a constant $c=c(d)$ such 
that for every $x,y\in \Omega$ with $|x-y|<c\, 
\mathrm{dist}(x,\partial \Omega)$ 
there exists a ball $B\in \mathcal{B}$ with $x,y\in B$.
 \item[(ii)] For every $B\in \mathcal{B}$, $B^\ast \subset \Omega$.
 \item[(iii)] The family $\{B^\ast\}_{B\in \mathcal{B}}$ has the finite 
overlapping property, i.e., each point of $\Omega$ belongs to at most $M=M(d)$ 
balls $B^\ast$.
\end{enumerate}
Thus for each $f\in H_k(\Omega)$,\
\begin{align}\label{6.14}
&\int_{\Omega\times \Omega}(f(x)-f(y))^2k(x,y)\dxy \nonumber \\
\geq\,\,&\frac{1}{M^2}\sum_{B\in \mathcal{B}} \int_{B^\ast \times 
B^\ast}(f(x)-f(y))^2k(x,y)\dxy \nonumber \\
\geq\,\,&\frac{c}{M^2}\sum_{B\in \mathcal{B}} \int_{B \times B} (f(x)-f(y))^2 
|x-y|^{-d-\alpha} \dxy \nonumber \\
\geq\,\,& \frac{c\tilde{c}}{M^2} \int_{\Omega \times 
\Omega}(f(x)-f(y))^2|x-y|^{-d-\alpha} \dxy,
\end{align}
where we applied inequality (13) in \cite[proof of Theorem 1]{Dyda06} to derive 
the last inequality, see also \cite[Theorem 1.6]{PrSaks17}.
For a scaled version of $\Omega$, we can
scale all balls in the family $\mathcal{B}$ by the same factor and arrive at
the same constant $\widetilde{c}$.
The constant stays bounded when $\alpha \in [\alpha_0,2)$ for $\alpha_0>0$.
\end{proof}

The next lemma follows from Lebesgue's 
differentiation theorem.
\begin{lemma}\label{lem:LebesgueDiff}
 Let $\varphi:\R^d\to \R$ be locally integrable. The following holds for 
almost every $s\in \R^d$. If $(x_h)_{h>0}$ is a sequence in $h\Z^d$ such that 
$s \in \widetilde{A}_h(x_h)$ for every $h>0$, then
\[
  \frac{1}{\lambda_d(A_h(x_h))} \int_{A_h(x_h)}\varphi(t)\, \d t
  \,\,\,
  \overset{h \to 0 }{\longrightarrow}
  \,\,\,
  \varphi(s).
\]
\end{lemma}

\begin{proof}
 The cube $\widetilde{A}_h(x_h)$ is contained in the ball $B_{2h\sqrt{d}}(s)$ 
and we know
$\lambda_d(\widetilde{A}_h(x_h))=c \lambda_d(B_
{2h\sqrt{d}}(s))$ for a constant $c$ only depending on the dimension $d$. 
Thus it follows by Lebesgue's differentiation theorem (cf 
\cite[Theorem 1.4 and Corollary 1.7 of Chapter 3]{Stein05}) for 
almost every $s\in \R^d$
\begin{align*} &\frac{1}{\lambda_d(A_h(x_h))}\int_{A_h(x_h)} 
|\varphi(t)-\varphi(s)|\, \d s \\ 
\leq\,\, &\frac{1}{\lambda_d(B_{2h\sqrt{d}}(s))}\int_{B_{2h\sqrt{d}}(s)} 
|\varphi(t)-\varphi(s)| \, \d s
  \,\,\,
  \overset{h \to 0 }{\longrightarrow}
  \,\,\,
  0.
\end{align*}
This implies our claim.
\end{proof}

\section*{Acknowledgement}
The authors thank Bartek Dyda for helpful discussions.


\def\cprime{$'$} \def\cprime{$'$}
  \def\polhk#1{\setbox0=\hbox{#1}{\ooalign{\hidewidth
  \lower1.5ex\hbox{`}\hidewidth\crcr\unhbox0}}}
  \def\polhk#1{\setbox0=\hbox{#1}{\ooalign{\hidewidth
  \lower1.5ex\hbox{`}\hidewidth\crcr\unhbox0}}}
  \def\polhk#1{\setbox0=\hbox{#1}{\ooalign{\hidewidth
  \lower1.5ex\hbox{`}\hidewidth\crcr\unhbox0}}} \def\cprime{$'$}
  \def\cprime{$'$} \def\cprime{$'$} \def\cprime{$'$} \def\cprime{$'$}

\end{document}